\documentclass[11pt]{amsart}
\usepackage{geometry}                
\usepackage{amsmath,amsxtra,amssymb,latexsym, amscd,amsthm}

\usepackage[colorlinks]{hyperref}
\usepackage{tikz}
\usepackage{xfrac}
\usepackage{nicefrac}
\usepackage{comment}
\usetikzlibrary{arrows}
\usepackage{bbm}
\usepackage{gensymb}
\usepackage{graphicx}
\usepackage{float}
\usepackage{cite}
\numberwithin{equation}{section}
\numberwithin{figure}{section}
\newtheorem{theorem}{Theorem}[section]
\newtheorem{lemma}[theorem]{Lemma}

\newtheorem{remark}[theorem]{Remark}

\begin{document}

\newcommand{\nn}{\mathbb N}
\newcommand{\cD}{\mathcal D}
\newcommand{\cR}{\mathcal R}
\newcommand{\cC}{\mathcal C}
\newcommand{\cB}{\mathcal B}
\newcommand{\cA}{\mathcal A}
\newcommand{\cM}{\mathcal M}
\newcommand{\cc}{\mathbb{C}}

\newcommand{\zz}{\mathbb Z}
\newcommand{\rr}{\mathbb R}
\newcommand{\ringz}{\mathring{\Sigma_{\mathbb Z}}}
\newcommand{\ve}{\varepsilon}
\newcommand{\tit}{\tilde{t}}
\newcommand{\tif}{\tilde{f}}

\newcommand{\tiD}{\widetilde{\mathcal D}}

\newcommand{\ld}{\lambda}
\newcommand{\til}{\tilde{\lambda}}
\newcommand{\tib}{\tilde{b}}
\newcommand{\tia}{\tilde{a}}
\newcommand{\one}{\mathbbm{1}_{I}(M_{\Lambda})}
\makeatletter
\def\blfootnote{\xdef\@thefnmark{}\@footnotetext}
\makeatother

\setcounter{secnumdepth}{3}
\title{Resonances for 1D half-line periodic operators: I. Generic case}
\author{Trinh Tuan Phong}

\maketitle

\begin{abstract}
The present paper addresses questions on resonances for \\
a $1$D Schr\"{o}dinger operator with truncated periodic potential. Precisely, we consider the half-line operator $H^{\nn}=-\Delta +V$ and $H^{\nn}_{L}= -\Delta + V\mathbbm{1}_{[0, L]}$ acting on $\ell^{2}(\mathbb N)$ with Dirichlet boundary condition at $0$ with $L \in \nn$. We describe the resonances of $H^{\nn}_{L}$ near the boundary of the essential spectrum of $H^{\nn}$ as $L \rightarrow +\infty$ in the generic case, hence complete the results introduced in \cite{klopp13} on the resonances of $H^{\nn}_{L}$.
\end{abstract}

\section{Introduction}\label{S:Intro}
\blfootnote{I would like to express my thanks to Fr\'{e}d\'{e}ric Klopp, my advisor, who walked me through the difficulties that I encountered while carrying out this work.}
 Let $V$ be a periodic potential of period $p$ and $-\Delta$ be the (negative) discrete Laplacian on $l^{2}(\mathbb Z)$. 
 We define the $1$D Schr\"{o}dinger operators $H^{\mathbb Z}:=-\Delta +V$ acting on $l^{2}(\mathbb Z)$:
 \begin{equation}\label{eq:9.1}
(H^{\mathbb Z}u)(n)= \left ((-\Delta +V)u\right) (n)= u(n-1)+u(n+1) +V(n)u(n), \;\; \forall n\in \mathbb Z
\end{equation} 
 and $H^{\mathbb N}:=-\Delta+V$ acting on $l^{2}(\mathbb N)$ with Dirichlet boundary condition (b.c.) at $0$.\\
 Denote by $\Sigma_{\mathbb Z}$ the spectrum of $H^{\mathbb Z}$ and $\Sigma_{\mathbb N}$ the spectrum of $H^{\mathbb N}$. One has the following description for the spectra of $H^{\bullet}$ where $\bullet \in \{\mathbb N, \mathbb Z\}$:  
\begin{itemize}
\item[$\bullet$]  $\Sigma_{\mathbb Z}$ is a union of disjoint intervals; the spectrum of $H^{\mathbb Z}$ is purely absolutely continuous (a.c.) and the spectral resolution can be obtained via the Bloch-Floquet decomposition (see \cite{pierre76} for more details).
\item[$\bullet$] $\Sigma_{\mathbb N}=\Sigma_{\mathbb Z} \cup \{v_i\}_{i=1}^m$ where $\Sigma_{\mathbb Z}$ is the a.c. spectrum of $H^{\mathbb N}$ and $\{v_i\}_{i=1}^m$ are isolated simple eigenvalues of $H^{\nn}$ associated to exponentially decaying eigenfunctions (c.f. \cite{pavlov94}).
\end{itemize}
Pick a large natural number $L$, we set:
$$ \text{$H^{\mathbb N}_L:=-\Delta+ V\mathbbm{1}_{[0, L]}$ on $l^2(\mathbb N)$ with Dirichlet boundary condition (b.c.) at $0$. }$$
It is easy to check that the operator $H^{\nn}_L$ is self-adjoint. Then, the resolvent $z\in \mathbb C^{+}\mapsto (z-H^{\nn}_L)^{-1}$ is well defined  on $\mathbb C^{+}$. Moreover, one can show that $z \mapsto (z-H^{\nn}_L)^{-1}$ admits a meromorphic continuation from $\mathbb C^{+}$ to $\mathbb C \backslash \left((-\infty, -2] \cup [2, +\infty)\right) $ with values in the self-adjoint operators from $l^2_{comp}$ to $l^2_{loc}$. Besides, the number of poles of this meromorphic continuation in the lower half-plane $\{Im E<0\}$ is at most equal to $L$ (c.f.  \cite[Theorem 1.1]{klopp13}). This kind of results is an analogue in the discrete setting for meromorphic continuation of the resolvents of partial differential operators (c.f. e.g. \cite{SZ91}).\\ 
Now, we define the \textit {resonances} of $H^{\nn}_L$, the main objet to study in the present paper, as the poles of the above meromorphic continuation. The resonance widths, the imaginary parts of resonances, play an important role in the large time behavior of wave packets, especially the resonances of the smallest width that give the leading order contribution (see \cite{SZ91} for an intensive study of resonances in the continuous setting and \cite{IK12, IK14, IK141, BNW05, korotyaev11} for a study of resonances of various $1$D operators).
\subsection{Resonance equation for the operator $H^{\nn}_{L}$}\label{Ss:equation}
Let $H_L$ be $H^{\mathbb N}_L$ restricted to $[0, L]$ with Dirichlet b.c. at $L$. Then, assume that 
\begin{itemize}
\item $\lambda_0 \leq \lambda_1\leq \ldots \leq \lambda_L$ are Dirichlet eigenvalues of $H_L$;
\item $a_k=a_k(L):=|\varphi_k(L)|^2$ where $\varphi_k$ is a normalized eigenvector associated to $\lambda_k$.
\end{itemize}
Then, resonances of $H^{\mathbb N}_L$ are solutions of the following equation (c.f. \cite[Theorem 2.1]{klopp13}): 
\begin{equation}\label{eq:resN}
S_L(E):=\sum_{k=0}^L \dfrac{a_k}{\lambda_k-E}=-e^{-i\theta(E)}, \qquad E= 2\cos\theta(E).
\end{equation} 
where the determination of $\theta(E)$ is chosen so that $\text{Im}\theta(E)>0$ and $\text{Re}\theta(E)\in (-\pi,0)$ when $\text{Im}E>0$.\\
Note that the map $E\mapsto \theta(E)$ can be continued analytically from $\cc^{+}$ to the cut plane $\cc\backslash ((-\infty, 2]\cup [2,+\infty))$ and its continuation is a bijection from $\cc\backslash ((-\infty, 2]\cup [2,+\infty))$ to $(-\pi, 0)+i\rr$. In particular, $\theta(E) \in (-\pi, 0)$ for all $E\in (-2, 2)$.\\
Taking imaginary parts of two sides of the resonance equation \eqref{eq:resN}, we obtain that    
\begin{equation}\label{eq:easyfact}
\text{Im} S_L(E)=\text{Im} E \sum_{k=0}^L \dfrac{a_k}{|\lambda_k-E|^2}=e^{\text{Im}E}\sin(\text{Re}\theta(E)).
\end{equation}
Note that, according to the choice of the determination $\theta(E)$, whenever $\text{Im} E>0$, $\sin(\text{Re}\theta(E))$ is negative and $\text{Im}S_L(E)>0$. Hence, all resonances of $H^{\mathbb N}_L$ lie completely in the lower half-plane $\{\text{Im} E<0\}$.

The distribution of resonances of $H^{\nn}_L$ in the limit $L\rightarrow +\infty$ was studied intensively in \cite{klopp13}. 
All results proved in \cite{klopp13} assume that the real part of resonances are far from the boundary point of the spectrum $\Sigma_{\zz}$ and far from the point $\pm 2$, the boundary of the essential spectrum of $-\Delta$. By "far", we mean the distance between resonances and $\partial \Sigma_{\mathbb Z} \cup \{\pm2\} $ is bigger than a positive constant independent of $L$.\\
In the present paper, we are interested in phenomena which can happen for resonances whose real parts are near $\partial \Sigma_{\zz}$ but still far from $\pm 2$. To study resonances below compact intervals in $\ringz$, the interior of $\Sigma_{\zz}$, the author in \cite{klopp13} introduced an analytical method to simplify and resolve the equation \eqref{eq:resN} (c.f. \cite[Theorems 5.1 and 5.2]{klopp13}). Unfortunately, such a method was efficient inside $\Sigma_{\zz}$ but does not seem to work near $\partial \Sigma_{\zz}$. Hence, a different approach is thus needed to study resonances near $\partial \Sigma_{\zz}$.

We observe from \eqref{eq:resN} that the behavior of resonances is completely determined by spectral data $(\ld_{k})_{k}, (a_{k})_{k}$ of $H_{L}$. As pointed out in \cite{klopp13}, the parameters $a_k$ associated to $\lambda_k \in \ringz$ near a boundary point $E_{0} \in \partial \Sigma_{\zz}$ can have two different behaviors depending on the potential $V$: Either $a_k \asymp \frac{1}{L}$ (the non-generic case) or $a_{k}$ is much smaller, $a_k \asymp  \frac{|\lambda_k-E_0|}{L}$ (the generic one). Each case requires a particular approach for studying resonances. \textit{In this paper, we only deal with the generic case}. 
\subsection{Resonances of $H^{\mathbb N}_L$ near $\partial \Sigma_{\zz}$ in the generic case} \label{Ss:results}
Pick $E_{0} \in \partial \Sigma_{\zz} \cap (-2, 2)$ and $L=Np+j$ where $p$ is the period of the potential $V$ and $0 \leq j \leq p-1$. To fix ideas, let's assume that $E_{0}$ is the left endpoint of a band $B_{i}=[E_{0}, E_{1}] \subset \Sigma_{\zz}$. For $L=Np+j$ large, we define $a^{0}_{p-1}(E), a_{j+1}(E)$ and $ d_{j+1}$ as in \eqref{eq:monodromy}, \eqref{eq:product} and Remark \ref{R:nearboundary}.\\ 
The generic case corresponds to the following generic condition 
$$\text{either} \left( a^{0}_{p-1}(E_{0}) \ne 0 \text{ and } d_{j+1}\ne 0\right) \text{or} \left( a^{0}_{p-1}(E_{0})=0 \text{ and } a_{j+1}(E_{0}) \ne 0 \right). \eqno{(G)} $$
Note that, by \cite[Lemma 4.2]{klopp13}, for $L=Np+j$ large,\\ 
$\partial \Sigma_{\zz} \cap \sigma(H_{L}) =\{E_{0}; a^{0}_{p-1}(E_{0})= a_{j+1}(E_{0})=0 \text{ and } b^{0}_{p}(E)\ne 0\}$. Therefore, in the generic case, $E_{0}$ is not an eigenvalue of $H_{L}$ when $L$ is large.\\    
Recall that all resonances whose real parts belong to a compact set in $(-2, 2) \cap \ringz$ were studied in \cite{klopp13}. In this paper, we will study the resonance equation \eqref{eq:resN} in the rectangle $\cD=\left[E_0, E_0+\ve_1\right] -i[0, \ve_{2}]$ where $\ve_{1}\asymp \ve^{2}$ and $\ve_{2} \asymp \ve^{5}$ with $\ve>0$ small. We will split $\cD$ into smaller rectangles $\{\cB_{n,\ve}\}$ with $0 \leq n \lesssim \ve L$ and study the existence, the uniqueness and the asymptotic of resonances in those rectangles (see Theorem \ref{T:maingeneric} and Figure \ref{F:f5}).

Following is our main idea to study the resonance equation \eqref{eq:resN} in the generic case. Let $H_L$, $(\ld_k)_{k}$ and $(a_k)_{k}$ be defined as in Subsection \ref{Ss:equation}. We figure out that, near the boundary of $\Sigma_{\zz}$, the spectral data $\{\ld_{k}\}$ and $\{a_{k}\}$ generically possess special properties. We exploit them to approximate $S_{L}(E)$. Concretely, for each eigenvalue $\ld_{k}$ of $H_{L}$ near $\partial \Sigma_{\zz}$, we approximate $S_{L}(E)$ in a domain close to $\ld_{k}$ by keeping the term $\frac{a_{k}}{\ld_{k}-E}$ and replacing the sum of the other terms by $\sum\limits_{\ell \ne k}\frac{a_{\ell}}{\ld_{\ell}-\ld_{k}}$. Then, we use the Rouch\'{e}'s theorem to describe resonances. For a domain farther from $\ld_k$, we make use of behavior of spectral data to show that there are no resonances there.\\
Surprisingly, this method has a flavor of the one used in describing resonances for operators with random potentials and resonances near isolated eigenvalues of $H^{\nn}$ for operators with periodic potentials (see \cite[Theorem 3.3]{klopp13}). Nonetheless, in the present case, the situation is different in many aspects. The hypotheses in \cite[Theorem 3.3]{klopp13} obviously do not hold in our case and a modification of the proofs in \cite{klopp13} could not be a solution either. Moreover, to obtain the asymptotic formulae for resonances in Theorem \ref{T:maingeneric}, it's crucial to study the regularity of spectral parameters (c.f. Section \ref{S:smoothness}). 

Our main result is the following:
\begin{theorem}\label{T:maingeneric}
Assume that $E_0 \in (-2, 2)$ is the left endpoint of the $i$th band $B_i$ of $\Sigma_{\zz}$. Let $H_L$, $(\ld_k)_{k}$ and $(a_k)_{k}$ be defined as in Subsection \ref{Ss:equation}.
We enumerate the spectral data $\lambda_{k}$ and $a_{k}$ in $B_i$ as $(\lambda^{i}_{\ell})_{\ell}, (a^{i}_{\ell})_{\ell}$ with $0 \leq \ell \leq n_{i}$ (the (local) enumeration w.r.t. the band $B_i$ of $\Sigma_{\zz}$).\\
Let $I=[E_{0}, E_{0}+\ve_{1}]$ and $\cD=\left[E_0, E_0+\ve_1\right] -i[0, \ve_{2}]$ where $\ve_{1}\asymp \ve^{2}$ and $\ve_{2} \asymp \ve^{5}$ with $\ve>0$ small.
Then, we have 
\begin{enumerate}
\item For each eigenvalue $\lambda^{i}_{n}\in I$ of $H_{L}$, there exists a unique resonance $z_{n}$ in 
$\mathcal B_{n, \ve}=\left[\frac{\lambda^{i}_{n-1}+\lambda^{i}_{n}}{2}, \frac{\lambda^{i}_{n}+\lambda^{i}_{n+1}}{2}\right] -i \left[0, \ve^5 \right]$ with a convention that $\lambda^{i}_{-1}:=2E_{0}-\lambda_{0}$. Moreover, $z_{n} \in \cM_{n}=\left [\frac{\lambda^{i}_{n-1}+\lambda^{i}_{n}}{2}, \frac{\lambda^{i}_{n}+\lambda^{i}_{n+1}}{2}\right] - i \left[0, C_0 \frac{n+1}{L^2} \right]$ with $C_{0}>0$ large. Besides, there are no resonances in the rectangle $[E_{0}-\ve, E_{0}] - i \left[0, C_0 \frac{n+1}{L^2} \right]$. 
\item Define $S^{i}_{n,L}(E) = S_{L}(E)- \frac{a^{i}_{n}}{\ld^{i}_{n}-E}$ and $\alpha_{n}= S^{i}_{n, L}(\ld^{i}_n)+e^{-i\theta(\ld^{i}_n)}$. Then, there exists $c_{0}>0$ s.t. $c_{0} \leq |\alpha_{n}| \lesssim \frac{1}{\ve^{2}}$ and  
\begin{equation}\label{eq:asymptotic*}
z_n =\ld^{i}_n +\frac{a^{i}_n}{\alpha_{n}} +O\left( \frac{(n+1)^4}{ L^5|\alpha_{n}|^{3}}\right).
\end{equation}
\item $\text{Im}z_{n}$ satisfies
\begin{equation}\label{eq:imEn*}
\text{Im}z_n =\frac{a^{i}_n \sin(\theta(\ld^{i}_n))}{\left|\alpha_{n}\right|^2} + O\left(\frac{(n+1)^4}{L^5 |\alpha_{n}|^{3}}\right).
\end{equation}
Consequently, there exists a large constant $C>0$ such that $ \frac{\ve^4 (n+1)^2}{C L^3} \leq  |\text{Im} z_n| \leq C \frac{(n+1)^{2}}{L^{3}}$. 
\end{enumerate}
\end{theorem} 
The above theorem calls for a few comments. First of all, near $\partial \Sigma_{\zz}$, each eigenvalue $\ld^{i}_n \in \ringz$ generates a unique resonance $z_n$ which can be described by the formula \eqref{eq:asymptotic*}. Moreover, $|\text{Im} z_n| \asymp \frac{n^2}{L^3}$. Hence, when $n \asymp \ve L$ (corresponding to resonances far from the boundary), $|\text{Im}z_{n}|$ is of order $\frac{1}{L}$ as stated in \cite[Proposition 1.4]{klopp13}. However, when $n$ is small i.e. $\ld^{i}_{n}$ is close to the boundary of the spectrum, $|\text{Im}z_{n}|$ becomes much smaller. Precisely, the magnitude of the width of resonances varies from $\frac{1}{L^{3}}$ to $\frac{1}{L}$. Combining our result with the description of resonances in \cite{klopp13}, we see that, below each band of $\Sigma_{\zz}$, resonances closest to the real axis form a nice curve whose shape is completely determined by the behavior of spectral parameters.
\begin{figure}[t]
\centering
\begin{tikzpicture}[line cap=round,line join=round,x=0.5cm,y=0.5cm]
\clip(-1.39,-1.0) rectangle (13.79, 9.13);
\draw (2,6)-- (10,6);
\draw (10,6)-- (10,4);
\draw (10,4)-- (2,4);
\draw (2,4)-- (2,6);
\draw (2,0)-- (10,0);
\draw (10,0)-- (10,4);
\draw (2,4)-- (2,0);
\draw (0.13,8.0) node[anchor=north west] {$\frac{\lambda^{i}_{n-1}+\lambda^{i}_n}{2}$};
\draw (8.12,8.0) node[anchor=north west] {$\frac{\lambda^{i}_n+\lambda^{i}_{n+1}}{2}$};
\draw (5.12,7.38) node[anchor=north west] {$\lambda^{i}_n$};
\draw (10.13,4.41) node[anchor=north west] {$-C_0 \frac{n+1}{L^2}$};
\draw (10.23,0.9) node[anchor=north west] {$-\varepsilon^5$};
\draw (6.0,5.42) node[anchor=north west] {$\mathcal M_n$};
\draw (6.11,2.41) node[anchor=north west] {$\mathcal A_{n,\varepsilon}$};
\begin{scriptsize}
\fill [] (2,6) circle (1.5pt);
\draw[] (2.36,6.3) node {$A$};
\fill [] (10,6) circle (1.5pt);
\draw[] (10.3,6.3) node {$B$};
\fill [] (10,4) circle (1.5pt);
\draw[] (10.34,4.59) node {$C$};
\fill [] (2,4) circle (1.5pt);
\draw[] (2.36,4.59) node {$D$};
\fill [] (2,0) circle (1.5pt);
\draw[] (2.32,0.57) node {$E$};
\fill [] (10,0) circle (1.5pt);
\draw[] (10.3,0.57) node {$F$};
\fill [] (6,6) circle (1.5pt);
\end{scriptsize}
\end{tikzpicture}
\caption[]{ Rectangle $\cB_{n,\ve}$}
\label{F:f5}
\end{figure}
Next, we turn our attention to resonances below $\rr \backslash \Sigma_{\nn}$. Recall that $\Sigma_{\nn}$, the spectrum of $H^{\nn}$, is the union of $\Sigma_{\zz}$ and isolated simple eigenvalues $\{v_{\ell}\}_{\ell}$ of the operator $H^{\nn}$.

Let $I$ be a compact interval in $(-2, 2)$ and $I \subset \rr \backslash \Sigma_{\nn}$. Then, by \cite[Theorem 1.2]{klopp13}, $H_{L}^{\nn}$ has no resonances in the rectangle $I-i[0, c]$ for some constant $c>0$. In other words, there exists a resonance free region of width at least of order $1$ below the compact interval $I$. This result is a direct consequence of the fact that $\text{dist}(E, \sigma(H_{L}))$ and $|\text{Im} e^{-i \theta(E)}|$ are lower bounded by a positive constant for $E\in I \subset (\rr \backslash \Sigma_{\nn}) \cap (-2, 2)$.\\
In this paper, we extend the above result for compact intervals $I$ which meet the boundary of $\Sigma_{\zz}$. 
\begin{theorem}\label{T:outside}
Let $E_{0} \in (-2, 2)$ be the left endpoint of the $i$-th band $B_{i}$ of $\Sigma_{\zz}$. Pick $L\in \nn^{*}$ large.
Then, $H^{\nn}_{L}$ has no resonances in the rectangle $[E_{0}-\ve, E_{0}]- i[0, \ve^{5}]$ with $\ve$ sufficiently small. 
\end{theorem}
Our paper is organized as follows. For each $\lambda^{i}_{n} \in I=[E_{0}, E_{0}+\ve_{1}]$ with $\ve_{1}\asymp \ve^{2}$, we study the resonance equation \eqref{eq:resN} in the rectangle\\ 
$\mathcal B_{n, \ve}=\left [\frac{\lambda^{i}_{n-1}+\lambda^{i}_{n}}{2}, \frac{\lambda^{i}_{n}+\lambda^{i}_{n+1}}{2}\right] -i \left[0, \ve^5 \right]$ (see Figure \ref{F:f5}). First of all, in Section \ref{S:sd}, we remind readers of the description of spectral data $\{\ld_{k}\}$ and $ \{a_{k}\}$ appeared in \cite{klopp13}. Next, we study the behavior of spectral data near the boundary point $E_0$ in Section \ref{S:smoothness}. Then, in Section \ref{S:sip}, we show that, in $\cA_{n,\ve}$ where $C_{0}\frac{n+1}{L^2} \leq |\text{Im}E| \leq \ve^5$ with $C_{0}>0$ sufficiently large, $|\text{Im}S_L(E)|$ will be too small. As a result, there are no resonances in  $\mathcal A_{n, \ve}$. After that, we will state the proof of Theorem \ref{T:maingeneric} in Section \ref{S:proofmain}. Finally, in the last section, Section \ref{S:proofmain1}, a proof for Theorem \ref{T:outside} is given.

\textbf{Notations:} Throughout the present paper, we will write $C$ for constants whose values can vary from line to line. Constants marked $C_{i}$ are fixed within a given argument. We write $a\lesssim b$ if there exists some $C>0$ independent of parameters coming into $a, b$ s.t. $a \leq Cb$. Finally, $a\asymp b$ means $a\lesssim b$ and $b\lesssim a$. 
\section{Spectral data-what are already known}\label{S:sd}
From \eqref{eq:resN}, resonances of $H_{L}^{\nn}$ depend only on the spectral data of the operator $H_{L}$ i.e., the eigenvalues and corresponding normalized eigenvectors of $H_{L}$. In order to "resolve" the resonance equation \eqref{eq:resN}, it is essential to understand how eigenvalues of $H_L$ behave and what the magnitudes of $a_l:=|\varphi_l(L)|^2$ are in the limit $L \rightarrow +\infty$.\\  
Before stating the properties of spectral data of $H_L$, one defines the \textit{quasi-momentum} of $H^{\mathbb Z}$:\\
Let $V$ be a periodic potential of period $p$ and $L$ be large. For $0\leq k\leq p-1$, one defines $\widetilde{T}_k=\widetilde{T}_k(E)$ to be a monodromy matrix for the periodic finite difference operators $H^{\mathbb Z}$, that is,
\begin{equation}\label{eq:monodromy}
\widetilde{T}_k(E)= T_{k+p-1, k}(E)=T_{k+p-1}(E)\ldots T_k(E)=
\begin{pmatrix}
 a^k_p(E) &a^k_p(E)\\
 a^k_{p-1}(E) &a^k_{p-1}(E)
\end{pmatrix}   
\end{equation}   
where $\{T_l(E)\}$ are transfer matrices of $H^{\mathbb Z}$: 
\begin{equation}\label{eq:transfer}
T_l(E)= 
\begin{pmatrix}
 E-V_l &-1\\
 1 &0
\end{pmatrix}.   
\end{equation}   
Besides, for $k\in \{0, \ldots, p-1\}$ we write
\begin{equation}\label{eq:product}
T_{k-1}(E\ldots T_0(E)= \begin{pmatrix}
a_k(E) & b_k(E)\\
 a_{k-1}(E) & b_{k-1}(E)
\end{pmatrix}.
\end{equation}
We observe that the coefficients of $\widetilde{T}_k(E)$ are monic polynomials in $E$. Moreover, $a^k_p(E)$ has degree $p$ and $b^k_p(E)$ has a degree $p-1$. The determinant of $T_l(E)$ equals to $1$ for any $l$, hence, $\det \widetilde{T}_k(E)=1$. Besides, $k \mapsto \widetilde{T}_k(E)$ is $p-$periodic since $V$ is a $p-$periodic potential. Moreover, for $j<k$
$$\widetilde{T}_k(E)= T_{k,j}(E) \widetilde{T}_j(E) T^{-1}_{k,j}(E).$$
Thus the discriminant $\Delta(E):= \text{tr} \widetilde{T}_k(E)=a^k_p(E)+b^k_{p-1}(E)$ is independent of $k$ and so are $\rho(E)$ and $\rho(E)^{-1}$, eigenvalues of $\widetilde{T}_k(E)$. Now, one can define the Floquet \textit{quasi-momentum} $E \mapsto \theta_p(E)$ by 
\begin{equation}\label{eq:2.5}
\Delta(E)=\rho(E)+\rho^{-1}(E)= 2\cos\left(p \theta_p(E)\right).
\end{equation}
Then, one can show that the spectrum of $H_{\nn}$, $\Sigma_{\mathbb Z}$, is the set $\{ E | |\Delta(E)|\leq 2 \}$ and $$\partial \Sigma_{\mathbb Z}=\{ E | |\Delta(E)|=2 \text{ and $\widetilde{T}_{0}(E)$ is not diagonal}\}.$$ 
Note that each point of $\partial \Sigma_{\zz}$ is a branch point of $\theta_p(E)$ of square-root type.
\begin{figure}[H]
\begin{center}
\begin{tikzpicture}[line cap=round,line join=round,x=0.67 cm,y=0.67 cm]
\clip(-2.76,-2.14) rectangle (6.47,4.08);
\draw[smooth,samples=100,domain=-2.758029237549914:6.47493127408347] plot(\x,{2*sin(((\x)+1.5)*180/pi)+0.09});
\draw (-4,2)-- (-4,0);
\draw [domain=-2.76:6.47] plot(\x,{(--2-0*\x)/2});
\draw (-4,0)-- (-4,-2);
\draw [domain=-2.76:6.47] plot(\x,{(-2-0*\x)/2});
\draw [line width=2.8pt] (-2.07,0)-- (-1.06,0);
\draw [line width=2.8pt] (1.17,0)-- (2.22,0);
\draw [line width=2.8pt] (4.21,0)-- (5.26,0);
\draw[->, >=latex] (0,-6)-- (0, 4.08);
\draw[->, >=latex](-6,0)-- (9,0);
\draw [->, >=latex] (5.26,0) -- (6.47,0);
\draw (0.06,4.21) node[anchor=north west] {$ \Delta(E) $};
\draw (5.71,0.04) node[anchor=north west] {$ E $};
\draw (0.11,1.69) node[anchor=north west] {$ 2 $};
\draw (-0.2,-0.88) node[anchor=north west] {$-2$};
\begin{scriptsize}
\draw (-3.82,-0.68) node {$E$};
\fill [] (-1.06,0) circle (1.5pt);
\draw[] (-0.86,0.33) node {$E^{+}_1$};
\fill [] (-2.07,0) circle (1.5pt);
\draw[] (-1.94,0.33) node {$E^{-}_1$};
\fill [] (1.17,0) circle (1.5pt);
\draw[] (1.36,0.33) node {$E^{-}_2$};
\fill [] (2.22,0) circle (1.5pt);
\draw[] (2.4,0.33) node {$E^{+}_2$};
\fill [] (4.21,0) circle (1.5pt);
\draw[] (4.4,0.33) node {$E^{-}_3$};
\fill [] (5.26,0) circle (1.5pt);
\draw[] (5.46,0.33) node {$E^{+}_3$};
\fill [] (0,0) circle (1.5pt);
\draw[] (0.2,0.33) node {$O$};
\end{scriptsize}
\end{tikzpicture}
\caption[Figure2]{Function $\Delta(E)$.}
\label{F:fk2}
\end{center}
\end{figure}
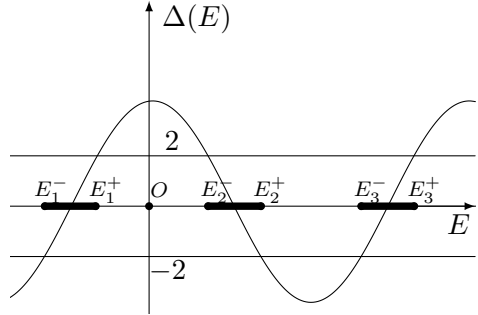
One decomposes $\Sigma_{\zz}$ into its connected components i.e.$\Sigma_{\zz}=\bigcup\limits_{i=1}^q B_i$ with $q<p$. Let $c_{i}$ be the number of closed gaps contained in $B_{i}$. Then, $\theta_{p}$ maps $B_{i}$ bijectively into $\sum\limits_{\ell=1}^{i-1} (1+c_{\ell}) \frac{\pi}{p} + \frac{\pi}{p}[0, c_{i}]$. Moreover, on this set, the derivative of $\theta_{p}$ is proportional to the common density of states $n(E)$ of $H^{\mathbb Z}$ and $H^{\mathbb N}$: 
$$ \theta'_p(E)= \pi n(E).$$
One has the following description for spectral data of $H_L$.
\begin{theorem}\cite[Theorem 4.2]{klopp13}\label{T:spectraldata}   
For any $j \in \{0, \ldots, p-1\},$ there exists $h_j: \Sigma_{\zz} \rightarrow \rr$ , a continuous function that is real analytic in a neighborhood of $\ringz$ such that, for $L=Np+j$,
\begin{enumerate}
\item The function $h_j$ maps $B_{i}$ into $\left(-(c_{i}+1)\pi, (c_{i}+1)\pi \right)$ where $c_i$ is the number of closed gaps in $B_i$;
\item the function $\theta_{p, L}= \theta_p - \frac{h_j}{L-j}$ is strictly monotonous on each band $B_{i}$ of $\Sigma_{\zz}$;
\item for $1\leq i \leq q$, the eigenvalues of $H_L$ in $B_{i}$, the $i$th band of $\Sigma_{\zz}$, says $(\lambda^{i}_k)_k$ are the solutions (in $\Sigma_{\zz}$) to the quantization condition
\begin{equation}\label{eq:quan1}
\theta_{p, L}(\lambda^i_k) =\dfrac{k \pi}{L-j}, \quad k\in \zz.
\end{equation}
\item If $\lambda$ is an eigenvalue of $H_L$ outside $\Sigma_{\zz}$ for $L=Np+j$ large, there exists $\lambda_{\infty} \in \Sigma^{+}_0 \cup \Sigma^{-}_j \backslash \Sigma_{\zz}$ s.t. $|\lambda-\lambda_{\infty}| \leq e^{-cL}$ with $c>0$ independent of $L$ and $\lambda$. 
\end{enumerate}
\end{theorem}
When solving the equation \eqref{eq:quan1}, one has to do it for each band $B_{i}$, and for each band and each $k$ such that $\frac{k \pi}{L-j} \in \theta_{p,L}(B_{i})$, \eqref{eq:quan1} admits a unique solution. But, it may happens that one has two solutions to \eqref{eq:quan1} for a given $k$ belonging to neighboring bands.\\
Following is the description of the associated eigenfunctions.
\begin{theorem}\cite[Theorem 4.2]{klopp13}\label{T:spectraldata1}
Recall that $(\lambda_{l})_{l}$ are the eigenvalues of $H_{L}$ in $\Sigma_{\zz}$ (enumerated as in Theorem \ref{T:spectraldata}).
\begin{enumerate}
\item There are $p+1$ positive functions, say, $f_0^{+}$ and $(f_j^{-})_{0 \leq j \leq p-1}$, that are 
real analytic in a neighborhood of $\ringz$ such that, for $E_0 \in \ringz$ such that $\theta_p(E_0) \in [\pi j/p, \pi(j+1)/p )$, there exists $\mathcal V$, a complex neighborhood of $E_0$, real constants $m_p,L$, $\kappa^0_{p,L}$ and $\kappa^j_{p,L}$ and a function $\Xi_{p,L}$ real analytic $\mathcal V$ such that, for $L=Np+k$ sufficiently large, for $\lambda_k \in \mathcal V$, one has
\begin{align} \label{eq:crazy}
|\varphi_l(0)|^2 = \frac{f_0(\lambda_l)}{L-k} F_{p,L,j_0}(\lambda_l), \qquad |\varphi_l(L)|^2= \frac{f_j(\lambda_l)}{f_0(\lambda_l)} |\varphi_l(0)|^2,\notag\\
\varphi_l(L)\overline{\varphi(0)}= \pm e^{i(L-j)\theta_{p,L}(\lambda_l)} \frac{\sqrt{f_0(\lambda_l)f_j(\lambda_l)}}{L-j} F_{p,L,j_0}(\lambda_j),\\
F_{p,L,j_0}(\lambda)= 1+\kappa g_{m_{p,L}-\pi j}((L-j)\theta_{p,L}(\lambda_l))+\frac{\Xi_{p,L}(\lambda)}{L-j} , \notag
\end{align}
where 
\begin{itemize}
	\item[$\bullet$] the constant $\kappa_{p,L}$ vanishes except if $E_0$ is a bad closed gap;
	\item[$\bullet$] for $m\ne \pi \zz$, the entire function $g_m$ is given by 
	$$g_m(z)=\frac{\sin^2(\pi z+m)}{(\pi z +m)\sin^2 m};$$
	\item[$\bullet$] for some $m_j(E_0) \in (-\pi, 0)$ and $\kappa_j(E_0) \in \rr$
	$$m_{p,L}=m_j(E_0)+ O\left(\frac{1}{L-j} \right) \text{ and } \kappa_{p,L}=\kappa_j(E_0)+O\left(\frac{1}{L-j} \right);$$
	\item[$\bullet$] the sign $\pm$ is constant on every band of the spectrum $\Sigma_{\zz}$.
\end{itemize} 
\item Let $\lambda$ be an eigenvalue outside $\Sigma_{\zz}$ and $\varphi$ is a normalized associated eigenvector, one has one of the following alternatives for $L=Np+j$ large 
\begin{itemize}
\item[$\bullet$] If $\lambda_{\infty} \in \Sigma_{\nn}\backslash \Sigma^{-}_j$, one has 
$$|\varphi(L)| \asymp e^{-cL} \text{ and } |\varphi(0)| \asymp 1;$$
\item[$\bullet$] If $\lambda_{\infty} \in \Sigma^{-}_j\backslash \Sigma_{\nn}$, one has 
$$|\varphi(L)| \asymp 1 \text{ and } |\varphi(0)| \asymp e^{-cL};$$
\item[$\bullet$] If $\lambda_{\infty} \in \Sigma_j^{-}\cap \Sigma_{\nn}$, one has 
$$ |\varphi(L)| \asymp 1 \text{ and } |\varphi(0)| \asymp 1.$$
\end{itemize}
Let $\mathcal B_{j}$ be the set of bad closed gaps for $j$ (c.f. \cite[Proposition 4.1]{klopp13} for more details). Then, there exists $\tilde{f}$ real analytic on $\ringz \backslash \mathcal B_{j}$ such that, for eigenvalues, on this set, in \eqref{eq:crazy}, one can take
$$F_{p,L,j_{0}}(\lambda) =  1+ \frac{\Xi_{p,L}(\lambda)}{L-j} = \left(1+ \frac{\tilde{f}(\lambda)}{L-j} \right)^{-1}.$$ 
\end{enumerate}
\end{theorem}
\begin{remark} \label{R:nearboundary}
According to \cite[Section 4]{klopp13}, one has the following behavior of $a_k$ associated to $\lambda_k$ which is close to $\partial \Sigma_{\zz}$.\\ Let $E_0 \in \partial \Sigma_{\zz}$ and $L=Np+j$. One defines $d_{j+1}=a_{j+1}(E_0) (a^0_p(E_0) -\rho^{-1}(E_0))+b_{j+1}(E_0) a^0_{p-1}(E_0)$ where $a_{j+1}, b_{j+1}, a^0_p, a^0_{p-1}$ are polynomials defined in \eqref{eq:monodromy} and \eqref{eq:product}. Then, one distinguishes two cases:    
\begin{enumerate}
	\item If $a^0_{p-1}(E_0)=0$, then 
	$$a_k=|\varphi_k(L)|^2 \asymp \frac{|\lambda_k-E_0|}{L-j} \text{ and } |\varphi_k(0)|^2 \asymp \frac{1}{L-j}.$$
	\item If $a^0_{p-1}(E_0)\ne 0$, then
	\begin{itemize}
		\item if $d_{j+1}\ne 0$, one has 
		$$|\varphi_k(L)|^2 \asymp \frac{|\lambda_k-E_0|}{L-j} \text{ and } |\varphi_k(0)|^2 \asymp \frac{|\lambda_k-E_0|}{L-j}.$$
		\item if $d_{j+1}=0$, one has
		$$|\varphi_k(L)|^2 \asymp \frac{1}{L-j} \text{ and } |\varphi_k(0)|^2 \asymp \frac{|\lambda_k-E_0|}{L-j}.$$	
	\end{itemize}	
\end{enumerate}
\end{remark}
\section[Behavior of spectral data]{Behavior of spectral data near the boundary of the spectrum}\label{S:smoothness}
Let $B_{i}$ be a band of $\Sigma_{\zz}$ and $(\lambda^{i}_{n})_{n}$ be (distinct) eigenvalues of $H_{L}$ in $B_{i}$. Pick $E_{0}$ an endpoint of $B_{i}$ and $L=Np+j$.\\ 
Then, according to Theorems \ref{T:spectraldata} and \ref{T:spectraldata1}, the spectral data i.e., the eigenvalues $\lambda^{i}_{\ell} \in B_i $ and the associated $a^{i}_{\ell}$, close to $E_{0}$, can be represented as $\lambda^{i}_{\ell}=\theta^{-1}_{p,L} \left( \frac{\pi \ell}{L-j}\right)$ if $\frac{\pi \ell}{L-j} \in \theta_{p,L}(B_i)$ and $a^{i}_{\ell}= \frac{1}{L-j} q(\lambda^{i}_{\ell})$ for some function $q$. In the present section, we will study the smoothness of the two functions $\theta^{-1}_{p,L}$ and $q$ near $\theta_{p,L}(E_{0})$ and $E_{0}$ respectively. We will make use of the results in this section to prove a good upper bound for the sum $S^{i}_{n,L}(\lambda^{i}_{n})$ defined in Theorem \ref{T:maingeneric} for each $\ld^{i}_{n} \in B_{i}$. Such an estimate plays an important role in describing the exact magnitude of the imaginary part of resonances near the boundary of $\Sigma_{\zz}$. Finally, we will study the behavior of eigenvalues in $B_{i}$ and close to the boundary point $E_{0}$.\\
Recall that, for $L=Np +j$, $\theta_{p,L}(E)= \theta_{p}(E)- \frac{h_{j}(E)}{L-j}$ where $\theta_{p}(E)$ is the quasi-momentum of $H^{\zz}$; $h_{j}(E)$ is analytic in $\ringz$ and satisfies the following relation
$$e^{2ih_{j}(E)} = \frac{a_{j+1}(E) \left( \rho(E)- a^{0}_{p}(E)\right) -b_{j+1}(E) a^{0}_{p-1}(E) }{\overline{a_{j+1}(E) \left( \rho(E)- a^{0}_{p}(E)\right) -b_{j+1}(E) a^{0}_{p-1}(E)}}.$$
For $0 \leq m \leq p-1$, we define $h_{m-1}(E)$ in the same way
$$e^{2ih_{m-1}(E)} = \frac{a_{m}(E) \left( \rho(E)- a^{0}_{p}(E)\right) -b_{m}(E) a^{0}_{p-1}(E) }{\overline{a_{m}(E) \left( \rho(E)- a^{0}_{p}(E)\right) -b_{m}(E) a^{0}_{p-1}(E)}}.$$
Here, $\rho(E)=e^{ip\theta_{p}(E)}$ and $a^{0}_{p-1}, a^{0}_{p}, a_{j+1}, b_{j+1}, a_{m}, b_{m}$ are polynomials defined in \eqref{eq:monodromy} and \eqref{eq:product}. \\
First of all, we prove the smoothness of $\theta^{-1}_{p,L}$.
\begin{lemma}\label{L:smoothness}
Let $E_0 \in \partial \Sigma_{\zz}$ and $B_i$ be the band of $\Sigma_{\zz}$ containing $E_0$. We put $J=\theta_{p,L}\left(B_i\right)$. Then, $\theta^{-1}_{p,L}$ is $ C^2$ on $J$ and its two first derivatives on $J$ are bounded by a constant independent of $L$. Besides, $\frac{d}{dx} \theta^{-1}_{p,L}(x)=0$ and
$\frac{d^2}{dx^2} \theta^{-1}_{p,L}(x) \ne 0$ at $x= \theta_{p,L}(E_{0})$.	
\end{lemma}
\begin{proof}[Proof of Lemma \ref{L:smoothness}]
Assume that $L=Np+j$ where $p$ is the period of the potential $V$ and $0 \leq j \leq p-1$. Since Theorem \ref{T:spectraldata}, $\theta_{p,L}$ is continuous and strictly monotonous on $B_i$. Hence, $J$ is a compact interval.\\
We can assume that $E_{0}$ is the left endpoint of the band $B_{i}$. Pick $x\in J$. Let $u\in B_i$ such that $\theta_{p,L}(u)=x$. Then, $\frac{d}{dx} \theta^{-1}_{p,L}(x)= \frac{1}{\theta'_{p,L}(u)}$. Note that $\theta_{p,L}(u)$ is analytic and strictly positive in the interior of the band $B_i$ (c.f. Theorem \ref{T:spectraldata}). Hence, it suffices to prove the lemma for $u$ near $E_0$.\\ 
It is well known that, for $u$ near $E_0$, $\theta_p'(u)=\pi n(u)=c_1|u-E_0|^{-1/2}(1+o(1))$ and $n'(u)=c_2 |u-E_0|^{-3/2}(1+o(1))$ where $n(u)$ is the density of state of the operator $H^{\zz}$ and $c_1, c_2 \ne 0$ (c.f. \cite{gerald00}).\\
Put $u=E_0+t^2$ with $t>0$. From the definition of $h_j$, we see that $t\mapsto h_{j}(E_{0}+t^{2})$ is analytic near $0$.
Indeed, we put $s(u)=a_{j+1}(u) \left( \rho(u)- a^{0}_{p}(u)\right) -b_{j+1}(u) a^{0}_{p-1}(u)$ with $\rho(u)=e^{ip\theta_{p}(u)}$. Then, since $\rho(E_{0}+t^{2})$ is analytic near $0$, so is $s(E_{0}+t^{2})$. We rewrite $s(E_{0}+t^{2})= \alpha(t) + i \beta(t)$ where $\alpha(t)=c_{\alpha}t^{k_{1}} \left(1+g_{1}(t) \right)$ and $\beta(t)=c_{\beta}t^{k_{2}} \left(1+g_{2}(t) \right)$ with $c_{\alpha}, c_{\beta} \ne 0$, $k_{1}, k_{2} \in \nn$ and $g_{1}(t), g_{2}(t)$ analytic near $0$. W.o.l.g., assume that $c_{\alpha}=c_{\beta}=1$ and $k_{1}\geq k_{2}$. Hence, $e^{2ih_{j}(E_{0}+t^{2})}$ is equal to
\begin{equation*}
\frac{t^{2(k_{1}-k_{2})}\left( 1+g_{1}(t)\right)^{2}- \left( 1+g_{2}(t)\right)^{2}+2i t^{2k_{1}}\left( 1+g_{1}(t)\right)\left( 1+g_{2}(t)\right)}{t^{2(k_{1}-k_{2})}\left( 1+g_{1}(t)\right)^{2}+\left( 1+g_{2}(t)\right)^{2} }. 
\end{equation*}
This formula implies directly that $h_{j}(E_{0}+t^{2})$ is analytic in $t$ near $0$.\\
Consequently, $h'_{j}(u)=O\left( |u-E_{0}|^{-1/2}\right)$ and $h^{''}_{j}(u)= O\left(|E-E_{0}|^{-3/2} \right) $.\\ 
Therefore, $\frac{d}{dx} \theta^{-1}_{p,L}(x)=c|u-E_0|^{1/2} (1+o(1))$ near $\theta_{p,L}(E_0)$ where $c \ne 0$. In the other words, $\frac{d}{dx} \theta^{-1}_{p,L}(x)$ is continuous and bounded by a constant independent of $L$ on $J$. Moreover, $\frac{d}{dx} \theta^{-1}_{p,L}(x)=0$ at $x=\theta_{p,L}(E_0)$.\\
Next, we study the second derivarive of $\theta^{-1}_{p,L}(x)$ on the interval $J$. We compute
\begin{equation}\label{eq:7.38}
\frac{d^2}{dx^2} \theta^{-1}_{p,L}(x)= - \frac{\theta''_{p,L}(u)}{ \left(\theta'_{p,L}(u) \right)^{3}}=
\frac{-\pi n'(u) + \frac{h_j^{''}(u)}{L-j}}{\left(\pi n(u)-\frac{h_j'(u)}{L-j} \right)^3 }.	
\end{equation}
We observe that the numerator of the right hand side (RHS) of \eqref{eq:7.38} is equal to $c_3|u-E_0|^{-3/2}(1+o(1))$ and the denominator of RHS of \eqref{eq:7.38} is equal to $c_4|u-E_0|^{-3/2}(1+o(1))$ near $E_0$ where $c_{3}, c_{4}$ are non-zero. Hence, $\theta^{-1}_{p,L}(x)$ is $C^{2}$ on the whole interval $J$ and $\frac{d^2}{dx^2} \theta^{-1}_{p,L}(x) \ne 0$ at $x=\theta_{p,L}(E_0)$. Moreover, its second derivative is bounded on $J$ by a positive constant independent of $L$.
\end{proof}
The following lemma will be useful for proving the smoothness of $a_{k}$ as a function of $\lambda_{k}$.
\begin{lemma}\label{L:analytic}
Let $E_{0}$ be one endpoint of a band $B_{i}$ of $\Sigma_{\zz}$. For $0\leq m \leq p-1$, we define, on the band $B_{i}$, 
\begin{equation}\label{eq:xim}
\xi_{m}(E)= \cos \left(h_{j}(E) - p\theta_{p}(E) -2h_{m-1}(E) \right) \cdot \frac{\sin \left(h_{j}(E) \right)}{\sin\left( p\theta_{p}(E)\right) }
\end{equation}
and $\nu_{m}(E)= 1-\cos \left(2h_j(E)-2h_{m-1}(E)\right)$.\\
Then, $\xi_{m}$ and $\nu_{m}$ are analytic near $0$ as a function of the variable $t= \sqrt{|E-E_{0}|}$. Moreover, $\nu_{m}(E)= \nu_{m,0} + \nu_{m,2}t^{2} +O(t^{3})$ where $\nu_{m,0}$ is equal to either $0$ or $2$, $\nu_{m,2} \in \rr$. 
\end{lemma}
\begin{proof}[Proof of Lemma \ref{L:analytic}]
W.o.l.g., we assume that $E_{0}$ is the left endpoint of $B_{i}$. Then, we write $E=E_{0}+t^{2}$ for $t>0$.
Let's consider the function $\xi_{m}(E)$ first. Note that $h_{j}(E_{0}+t^{2}), h_{m-1}(E_{0}+t^{2})$  are analytic in $t$ near $0$ (see the proof of Lemma \ref{L:smoothness}). Hence, we can write 
\begin{align}\label{eq:7.57}
h_{j}(E_{0}+t^{2}) = h_{j}(E_{0}) + h_{j,1}t +h_{j,2}t^{2} + O(t^{3}) ;\\
h_{m-1}(E_{0}+t^{2}) = h_{m-1}(E_{0}) + h_{m-1,1}t +h_{m-1,2}t^{2} + O(t^{3})
\end{align}
where $h_{j}(E_{0}), h_{m-1}(E_{0})$ belong to either $\pi \mathbb Z$ or $\frac{\pi}{2}+\pi \zz$ (see \cite[Lemma 4.4]{klopp13}).\\
\textit{First case: Assume that $h_{j}(E_{0}) \in \pi \zz$.}\\ 
By Taylor's expansion for the sine function, we have
\begin{equation}\label{eq:7.58}
\sin \left(h_{j}(E) \right)= 
\pm \left(h_{j,1}t+h_{j,2}t^{2}\right) +O(t^{3}).  
\end{equation}
W.o.l.g., assume that $\sin \left(h_{j}(E) \right)= h_{j,1}t+h_{j,2}t^{2} +O(t^{3}).$\\  
Next, we write $\theta_{p}(E) = \theta_{p, 0}+ \theta_{p, 1}t  + O(t^{2})$ where $\theta_{p, 0}=\theta_{p}(E_{0}) \in \frac{\pi}{p} \zz$ and $\theta_{p, 1} \ne 0$. Note that $p\theta_{p}(E_{0}), 2h_{m-1}(E_{0}) \in \pi \zz$, hence, $h_{j}(E_{0})-p\theta_{p}(E_{0})-2h_{m-1}(E_{0}) \in \pi \zz$. Then, Taylor's expansion for the cosine function yields
\begin{align}\label{eq:7.59}
\cos \left(h_{j}(E) - p\theta_{p}(E) -2h_{m-1}(E) \right) &= \pm 1  \mp \frac{t^{2}}{2} (h_{j,1}-p\theta_{p,1} -2h_{m-1, 1})^{2}+O(t^{3}).
\end{align} 
Hence, \eqref{eq:7.58}-\eqref{eq:7.59} yield 
\begin{align}\label{eq:7.60}
\cos \left(h_{j}(E) - p\theta_{p}(E) -2h_{m-1}(E) \right) \cdot \sin \left(h_{j}(E) \right) = \epsilon h_{j,1}t + \epsilon h_{j,2} t^{2}+ O(t^{3}) 
\end{align}
where $\epsilon=\pm 1$.\\
On the other hand, since $p\theta_{p}(E_{0}) \in \pi \zz$, we have
\begin{equation}\label{eq:thetap}
\sin \left( p \theta_{p}(E)\right) = \pm p \theta_{p,1} t + O(t^{3}).
\end{equation}
Thanks to \eqref{eq:7.60}-\eqref{eq:thetap}, we infer that
\begin{equation}\label{eq:7.61}
\xi_{m}(E)= \cos \left(h_{j}(E) - p\theta_{p}(E) -2h_{m-1}(E) \right) \cdot \frac{\sin \left(h_{j}(E) \right)}{\sin\left( p\theta_{p}(E)\right) } = \xi_{m,0} + \xi_{m,1}t +O(t^{2})
\end{equation}
where $\xi_{m,0}$ and $\xi_{m,1}$ are independent of  $m$.\\ 
Note that, all functions which we have considered so far are analytic in $t$. Hence, $O(t^{2})$ in \eqref{eq:7.61} can be written in $t^{2}g(t)$ where $g(t)$ is analytic near $0$. Hence, the above asymptotic shows that $\xi_{m}$ is analytic near $0$ as a function of $t$.\\
\textit{Second case: $h_{j}(E_{0}) \in \frac{\pi}{2}+ \pi \zz$.}\\
Note that, when we use the Taylor expansions in this case, the roles of the sine and the cosine terms in $\xi_{m}(E)$ are interchanged. Hence, $\xi_{m}(E)$ can be written in the same form as in \eqref{eq:7.61} but with different coefficients $\xi_{m,0}$ and $\xi_{m,1}$ which can depend on $m$. Hence, $t \mapsto \xi_{m}(E_{0}+t^{2})$ is analytic near $0$.\\
Finally, we consider the function $\nu_{m}(E)$. Since $h_{j}, h_{m-1}$ are analytic in $t$ near $0$, so is $\nu_{m}$. On the other hand, $2h_{j}(E_{0}), 2h_{m-1}(E_{0})$ always belong to $\pi \zz$. Hence, from Taylor's expansion of the function cosine, it is easy to see that the the coefficient of order $1$ in Taylor's expansion of $\nu_{m}(E_{0}+t^{2})$ vanishes. Precisely, $\nu_{m}(E)= \nu_{m,0} + \nu_{m,2}t^{2} +O(t^{3})$ where $\nu_{m,0}$ is equal to either $0$ or $2$.
\end{proof}
Now we prove the smoothness of $a_{k}$ as a function of $\lambda_{k}$.
\begin{lemma}\label{L:nicefunction}
Let $E_{0}$ be an endpoint of one band $B_{i}$ of $\Sigma_{\zz}$ and $L=Np+j$. Assume the condition (G) (see Subsection \ref{Ss:results}). Then, for each eigenvalue $\lambda_{k}$ near $E_{0}$, $a_{k}= \frac{1}{L-j} q(\lambda_{k})$ where $q$ is a $C^{1}$ function near $E_{0}$ and $q$ is bounded near $E_{0}$ by a constant $C_{\text{lip}}$ independent of $L$.\\
Consequently, for all $\lambda_{k}, \lambda_{n}$ near $E_{0}$, we have
\begin{equation}\label{eq:lip}
|a_k-a_n| \leq \frac{C_{\text{lip}}}{L}|\lambda_k-\lambda_n|.
\end{equation}   	
\end{lemma}
\begin{proof}[Proof of Lemma \ref{L:nicefunction}]
Pick an eigenvalue $\lambda_{k}\in \ringz$ close to $E_{0}$. Then, according to  Theorem \ref{T:spectraldata1} and \cite[Lemma 4.6]{klopp13}, 
 we have 
$a_k = \frac{1}{L-j} q(\lambda_k)$ where $L=Np+j$ and $q(E)$ is equal to 
\begin{align*}
 \frac{|a^0_{p-1}(E)|^2}{\left |a_{j+1}(E) \left(a^{0}_{p}(E)-\rho^{-1}(E) \right) + b_{j+1}(E)a^{0}_{p-1}(E) \right|^2 } f^{-1}(E) \left(1+\frac{\tilde{f}(E)}{L-j}\right)^{-1}
\end{align*} 
To prove the present lemma, it suffices to show that the derivative $q'(E)$ is continuous up to $E_0$ and $q'(E)$ is bounded near $E_0$ by a constant independent of $L$. Note that our function $q(E)$ depends on $L$.\\
We define, for $0\leq m \leq p-1$, 
$$v_{m}(E)= a_{m}(E) \left(  a^{0}_{p}(E)-\rho^{-1}(E)\right) +b_{m}(E)a^{0}_{p-1}(E)$$ 
and  $\psi_{m}(E)=\text{Re}\left(v_{m}(E) \right)= a_{m}(E) \left(a^{0}_{p}(E)-\cos(p\theta_{p}(E)) \right)  +b_{m}(E)a^{0}_{p-1}(E)$.\\ 
Since $\theta_{p}(E)-\theta_{p}(E_{0})= c_{1}|E-E_{0}|^{1/2}(1+o(1))$,\\ 
$\theta'_{p}(E)=\pi n(E)=c_{2}|E-E_{0}|^{-1/2}(1+o(1))$ with $c_{1}, c_{2} \ne 0$ and $\sin\left(p \theta_{p}(E_{0}) \right)=0$, we infer that $\psi_{m}(E)$  is a $C^{1}$ function up to $E_{0}$.\\ 
Consequently, $|v_{m}(E)|^{2} =\psi^{2}_{m}(E) +a^{2}_{m}(E) \sin^{2}\left( p\theta_{p}(E) \right)$ is $C^{1}$ on the band $B_{i}$ of $\Sigma_{\zz}$.\\
W.o.l.g., assume that $E_{0}$ is the left endpoint of the band $B_{i}$. We write $E=E_{0}+t^{2}$ for $t>0$. Then, $\rho(E_{0}+t^{2}), \theta_{p}(E_{0}+t^{2})$ are analytic in $t$ near $0$.\\
Recall that\\ 
$f(E)=\frac{2}{p} \sum\limits_{m=0}^{p-1}|\alpha_m(E)|^2$ where $\alpha_m(E)=\frac{v_m(E)}{\rho(E)-\rho^{-1}(E)}= -i \frac{v_m(E)}{2\sin(p\theta_p(E))}$ and  
\begin{align}\label{eq:bosung}
\tilde{f}(E)&=\frac{2}{f(E)} \left[ \sum_{m=0}^{p-1} |\alpha_{m}(E)|^{2} \cos \left(h_{j}(E) - p\theta_{p}(E) -2h_{m-1}(E) \right)  \right] \cdot \frac{\sin \left(h_{j}(E) \right)}{\sin\left( p\theta_{p}(E)\right) }  \notag\\
&+ \frac{2}{f(E)} \sum_{m=0}^{j} |\alpha_{m}(E)|^{2} \left(1- \cos\left(2 \left(h_{j}(E)-h_{m-1}(E) \right)  \right) \right) \\
&= \frac{2}{f(E)} \sum_{m=0}^{p-1} |\alpha_{m}(E)|^{2}\xi_{m}(E) +\frac{2}{f(E)} \sum_{m=0}^{j} |\alpha_{m}(E)|^{2}\nu_{m}(E) \notag.
\end{align}
(c.f. \cite[Section 4.1.4]{klopp13}).\\
We compute
\begin{equation}\label{eq:7.51}
|\alpha_m(E)|^2=\frac{1}{4\sin^2(p\theta_p(E))}  \left(a_m^2(E)\sin^2(p\theta_p(E)) + \psi_m^2(E) \right)  .
\end{equation}
Recall that, since \eqref{eq:thetap}, we can represent $\sin \left( p \theta_{p}(E)\right) = \pm p \theta_{p,1} t + O(t^{3})$  with $\theta_{p, 1} \ne 0$.\\  
Next, since $a_{m}(E)$ is a polynomial and $\psi_{m}(E)$ is $C^{1}$ up to $E_{0}$, we can write 
\begin{equation}\label{eq:7.53}
a_{m}(E_{0}+t^{2})= a_{m}(E_{0}) + a'_{m}(E_{0})t^{2} + O(t^{4});
\end{equation}
\begin{equation}\label{eq:7.54}
\psi_{m}(E_{0}+t^{2}) = \psi_{m}(E_{0}) + \psi_{m,2}t^{2} + O(t^{3}).
\end{equation}
Plugging \eqref{eq:thetap}, \eqref{eq:7.53} and \eqref{eq:7.54} in \eqref{eq:7.51}, we obtain
\begin{equation}\label{eq:7.55}
|\alpha_m(E)|^2= \frac{1}{t^{2}} \left( \alpha_{m, 0} + \alpha_{m,2}t^{2} + O(t^{3})  \right) 
\end{equation}
where $\alpha_{m,0} = \frac{\psi^{2}_{m}(E_{0})}{4 p^{2} \theta^{2}_{p,1}}$, $\alpha_{m,2} =  \frac{1}{4 p^{2} \theta^{2}_{p,1}} \left(
p^{2}\theta^{2}_{p,1}a^{2}_{m}(E_{0}) + 2 \psi_{m}(E_{0}) \psi_{m,2}\right) $.
Hence,
\begin{equation}\label{eq:7.56}
f(E_{0}+t^{2})= \frac{1}{t^{2}} \left(f_{0} + f_{2}t^{2} +O(t^{3}) \right) 
\end{equation}
where $f_{0}= \frac{1}{2p^{3} \theta^{2}_{p,1}}  \sum\limits_{m=0}^{p-1}\psi^{2}_{m}(E_{0})$ and 
$$f_{2}= \frac{1}{2p^{3} \theta^{2}_{p,1}} \left(
p^{2}\theta^{2}_{p,1} \sum_{m=0}^{p-1}a^{2}_{m}(E_{0}) + 2 \sum_{m=0}^{p-1}\psi_{m}(E_{0}) \psi_{m,2}\right). $$
For $0\leq m \leq p-1$, let $\xi_{m}(E)$ be the function defined in \eqref{eq:xim}, Lemma \ref{L:analytic}. Then, $\xi_{m}(E_{0}+t^{2})$ is analytic in $t$ near $0$ and we write $\xi_{m}(E) =\xi_{m,0}+\xi_{m,1} t + O(t^{2})$.
Hence, \eqref{eq:7.55} yield
\begin{equation}\label{eq:7.62}
2 \sum_{m=0}^{p-1}|\alpha_m(E)|^2 \xi_{m}(E) = \frac{1}{t^{2}} \left( \beta_{0} + \beta_{1} t + O(t^{2}) \right) 
\end{equation} 
where $\beta_{0}= 2\sum\limits_{m=0}^{p-1} \alpha_{m,0}\xi_{m,0}$ and $\beta_{1}=2\sum\limits_{m=0}^{p-1} \alpha_{m,0}\xi_{m,1}$.\\
Note that all series expansions in $t$ which we have used so far are associated to analytic functions in $t$. We thus can write $O(t^{2})$ in \eqref{eq:7.62} as $\beta_{2} t^{2}+t^{3}g_{0}(t)$ for some $\beta_{2}\in \mathbb R$ and $g_{0}$ analytic near $0$.\\
Now we put $\nu_{m}(E)= 1-\cos \left(2h_j(E)-2h_{m-1}(E)\right)$. Thanks to Lemma \ref{L:analytic}, we can write $\nu_{m}(E)= \nu_{m,0} + \nu_{m,2}t^{2} +O(t^{3})$. Then,
\begin{equation}\label{eq:7.63}
2 \sum_{m=0}^{j}|\alpha_m(E)|^2 \nu_{m}(E) = \frac{1}{t^{2}} \left( \gamma_{0} + \gamma_{2} t^{2} + O(t^{3}) \right) 
\end{equation} 
where $\gamma_{0}= 2\sum\limits_{m=0}^{j} \alpha_{m,0}\nu_{m,0}$.\\
To sum up, 
\begin{equation*}
f(E) \tilde{f}(E)= \frac{1}{t^{2}} \left( \beta_{0}+\gamma_{0} +\beta_{1}t +(\beta_{2}+\gamma_{2})t^{2} + t^{3}g_{1}(t)  \right) 
\end{equation*}
with $g_{1}(t)$ analytic near $0$.\\
Hence, for $L=Np+j$ large, we have
\begin{align}\label{eq:7.64}
&f(E)+ \frac{1}{L-j}f(E) \tilde{f}(E) = \notag \\ 
&\frac{1}{t^{2}} \left[ \left( f_{0}+\frac{\beta_{0}+\gamma_{0}}{L-j}\right) + \frac{\beta_{1}}{L-j}t + \left( f_{2} + \frac{\beta_{2}+\gamma_{2}}{L-j}\right) t^{2}  +t^{3}g_{2}(t) 
    \right]
\end{align}
where $g_{2}(t)$ is analytic near $0$. Moreover, $g_{2}(t)$ and its derivative are bounded by a constant independent of $L$ .\\
According to the condition (G), we distinguish two cases:\\ 
\textit{First case: Assume that $a^0_{p-1}(E_0)\ne 0$ and $d_{j+1}= v_{j+1}(E_{0})\ne 0$.}\\
First of all, in the present case, the function 
$$\frac{|a^0_{p-1}(E)|^2}{\left |a_{j+1}(E) \left(a^{0}_{p}(E)-\rho^{-1}(E) \right) + b_{j+1}(E)a^{0}_{p-1}(E) \right|^2 }= \frac{|a^0_{p-1}(E)|^2}{\left|v_{j+1}(E) \right|^{2}}$$ 
is analytic in $t$ near $0$.\\
Next, from the definition of $a_{m}(E)$ and $b_m(E)$ (see \eqref{eq:product}, Section \ref{S:sd}), we have  
\begin{equation}\label{eq:7.66}
	\begin{vmatrix}
	a_{m}(E) & b_{m}(E)\\
	a_{m-1}(E) & b_{m-1}(E)
	\end{vmatrix}
	=1.
\end{equation}
Combining this with the hypothesis that $a^0_{p-1}(E_0)\ne 0$, we infer that there exists $m_{0} \in \{0, \ldots, p-1 \}$ s.t. $\psi_{m}(E_{0}) \ne 0$. Hence, $f_{0}>0$. Then, with $L=Np+j$ sufficiently large, we can represent
\begin{equation}\label{eq:7.65}
f^{-1}(E) \left(1+\frac{\tilde{f}(E)}{L-j}\right)^{-1} = \frac{t^{2}}{f_{0,L}+ tg_{3}(t)}
\end{equation}
where the analytic function $g_{3}(t)$ and its derivative are bounded by a positive constant independent of $L$. Moreover, $f_{0,L}$ is lower bounded and upper bounded by positive constants independent of $L$. Hence, $q(E)$ can be written as $t^{2}g_{L}(t)$ where $g_{L}(t)$ is analytic near $0$. The function $g_{L}(t)$ does not vanish at $0$ and $\max \{|g_{L}(t), g'_{L}(t)| \} \leq C$ with some $C>0$ independent of $L$.\\ 
\textit{Second case: Assume that $a^0_{p-1}(E_0)=0$ and $a_{j+1}(E_{0}) \ne 0$}\\
Since $a^0_{p-1}(E_0)=0$, the monodromy matrix $\widetilde{T}_{0}(E_{0})$ defined in \eqref{eq:monodromy} is upper triangular with eigenvalues $\rho(E_{0})=\rho^{-1}(E_{0})=\pm 1$. Hence, $a^{0}_{p}(E_{0}) -\rho^{-1}(E_{0})=0$. Then, $\psi_{m}(E_{0})$, hence $\alpha_{m,0}$, is equal to $0$ for every $0\leq m\leq p-1$. Consequently, $f_{0}=\beta_{0}=\gamma_{0}=\beta_{1}= 0$. On the other hand, in the present case, $f_{2}= \frac{1}{2p} \sum_{m=0}^{p-1} a^{2}_{m}(E_{0})$. Note that, by \eqref{eq:7.66}, $a_{m}(E)$ and $a_{m-1}(E)$ can not vanish at $E_{0}$ simultaneously. Therefore, $f_{2}$ is strictly positive and 
\begin{equation}\label{eq:7.67}
f^{-1}(E) \left(1+\frac{\tilde{f}(E)}{L-j}\right)^{-1}= \frac{1}{f_{2,L} + t g_{4}(t)}
\end{equation}
where $g_{4}(t)$ is analytic near $0$. Moreover, there exists $C>0$ independent of $L$ such that $\max \{ |g_{4}(t)|, |g'_{4}(t)|  \} \leq C$ near $0$ and $\frac{1}{C}\leq f_{2, L}\leq C$.\\
Now we study the function 
$$\frac{|a^0_{p-1}(E)|^2}{\left |a_{j+1}(E) \left(a^{0}_{p}(E)-\rho^{-1}(E) \right) + b_{j+1}(E)a^{0}_{p-1}(E) \right|^2 }= \frac{|a^0_{p-1}(E)|^2}{\left|v_{j+1}(E) \right|^{2}}.$$ 
We rewrite $a^{0}_{p}(E)-\rho^{-1}(E)=a^{0}_{p}(E)-\cos \left( p \theta_{p}(E)\right) + i \sin \left( p \theta_{p}(E)\right)$.\\
We put $\varphi(E)=a^{0}_{p}(E)-\cos(p\theta_{p}(E))$. We observe that $\varphi(E)$ is a $C^{1}$ function up to $E_{0}$. On the other hand $\varphi(E_{0})=0$ since $a^{0}_{p}(E_{0}) -\rho^{-1}(E_{0})=0$. Hence, $\varphi(E) = (E-E_{0}) \left(\varphi'(E_{0})+o(1)\right)$. Note that $\sin \left( p \theta_{p}(E)\right)= c|E-E_{0}|^{1/2}(1+o(1))$ with $c\ne 0$. Therefore, $a^{0}_{p}(E)-\rho^{-1}(E)= \tilde{c}|E-E_{0}|^{1/2}(1+o(1))$ with $\tilde{c}\ne 0$.\\ 
On the other hand, $a_{j+1}(E_{0}) \ne 0$ and $a^{0}_{p-1}(E)$ has only simple roots. Hence, $\frac{|a^0_{p-1}(E)|^2}{\left|v_{j+1}(E) \right|^{2}}= t^{2} g_{5}(t)$ where $g_{5}(t)$ is analytic near $0$ and $g_{5}(0)\ne 0$.\\
Combining this with \eqref{eq:7.67}, there exist $C>0$ independent of $L$ and an analytic function $g_{L}(t)$ near $0$, $g_{L}(0) \ne 0$ such that \\
$q(E)= t^{2} g_{L}(t)$ and $\max \{|g_{L}(t)|, |g'_{L}(t)| \} \leq C$ near $0$.\\
\textbf{To sum up, in both cases, $q(E)=t^{2} g_{L}(t)$}.\\ 
This implies directly that $a_{k} \asymp \frac{|\lambda_{k}-E_{0}|}{L-j}$ in the generic case. Moreover, $q'(E) = g_{L}(t) + \frac{t}{2} g'_{L}(t)$ where $E=E_{0}+t^{2}$ and $t>0$. Hence, $q(E)$, as a function of $E$, is $C^{1}$ up to $E_{0}$. Besides, its derivative near $E_{0}$ is bounded by a constant $C_{\text{lip}}$ independent of $L$. As a result, \eqref{eq:lip} follows and we have the lemma proved. 
\end{proof}
\begin{remark}\label{R:becomeeigenvalue}
Assume that the boundary point $E_{0}$ satisfies the condition\\ 
$a^0_{p-1}(E_0)=a_{j+1}(E_{0})=0$. Then, since \cite[Lemma 4.2]{klopp13} and the fact that the matrix $ \widetilde{T}_{0}(E_{0})$ defined in \eqref{eq:monodromy} is not diagonal,   
 $E_{0}$ is an eigenvalue of $H_{L}$ for all $L$ large. Note that the hypothesis $a_{j+1}(E_{0})=0$ and \eqref{eq:7.66} imply that $b_{j+1}(E_{0}) \ne 0$. Hence, $\frac{|a^0_{p-1}(E)|^2}{\left|v_{j+1}(E) \right|^{2}}= c(1+o(1)) $ near $E_{0}$ with $c\ne 0$. Combining this with \eqref{eq:7.67}, we infer that $a_{k} \asymp \frac{1}{L}$ for all eigenvalues $\lambda_{k}$ of $H_{L}$ close to $E_{0}$. Besides, it is not hard to check that, from the proof of Lemma \ref{L:nicefunction}, we find again the behavior of $a_k$ stated in Remark \ref{R:nearboundary} for both cases, the generic and non-generic one.
\end{remark}
Finally, we state and prove an estimate for eigenvalues of $H_L$ in the band $B_i=[E_0, E_1]$ of $\Sigma_{\zz}$ and close to the boundary point $E_0$.   
\begin{lemma}\label{L:asymptoticformula}
Let $E_0 \in (-2, 2)$ be the left endpoint of the $i$th band $B_i$ of $\Sigma_{\zz}$. Let 
$\lambda^{i}_0 < \lambda^{i}_1< \ldots < \lambda^{i}_{n_i}$ be eigenvalues of $H_{L}$ in $\mathring{B_i}$, the interior of $B_{i}$.\\
Pick $\ve>0$ a small, fixed number and $\ve_1\asymp \ve^2$. Let $I=I_{\ve_1}: =[E_0, E_0+\ve_1] \subset (-2,2) \cap \Sigma_{\mathbb Z}$.\\
Assume that $\lambda^{i}_k$ is an eigenvalue of $H_L$ in $I$. 
Then, $k\leq \ve (L-j)$ and $\lambda^{i}_k-E_0 \asymp \frac{(k+1)^2}{L^2}$ (for $k\geq 1$, we will write $\lambda^{i}_k-E_0 \asymp \frac{k^2}{L^2}$ instead).\\ 
Moreover, there exists $\alpha>0$ s.t. for any $0 \leq n <k \leq \ve (L-j)$, we have
\begin{equation}\label{eq:survive}
\frac{|k^2-n^2|}{\alpha L^2}  \leq |\lambda^{i}_k -\lambda^{i}_n| \leq \frac{\alpha |k^2-n^2|}{L^2}.
\end{equation}
\end{lemma}
\begin{proof}[Proof of Lemma \ref{L:asymptoticformula}]
To simplify notations, we will skip the superscript $i$ in $\lambda^i_k$ of $H_L$ throughout this proof.\\ 
First of all, from the property of $\theta_{p}$ and $h_{j}$ near $E_{0}$, we have, for any $E$ near $E_{0}$, 
\begin{equation}\label{eq:114}
\theta_{p,L}(E)-\theta_{p,L}(E_{0})= c(L)\sqrt{|E-E_{0}|} \left( 1+o(1)\right) 
\end{equation}
where $|c(L)|$ is lower bounded and upper bounded by positive constants independent of $L$.\\
Put $L=Np+j$ where $p$ is the period of the potential $V$ and $0 \leq j \leq p-1$. According to Theorem \ref{T:spectraldata}, $\theta_{p,L}(E)$ is strictly monotone on $B_{i}$. W.o.l.g., we assume that $\theta_{p,L}(E)$ is strictly increasing on $B_{i}$.
Note that, in this lemma, we enumerate eigenvalues $\ld_{\ell}$ in $\mathring{B_{i}}$ with the index $\ell$ starting from $0$. Then, we have to modify the quantization condition \eqref{eq:quan1} in Theorem \ref{T:spectraldata} appropriately. Recall that the quantization condition is $\theta_{p,L}(\ld_{\ell})= \frac{\pi \ell}{L-j}$ where $\frac{\pi \ell}{L-j} \in \theta_{p,L}(B_{i})$ with $\ell \in \zz$. Assume that $\theta_{p}(E_{0})= \frac{m \pi}{p}$ with $m \in \zz$.\\ 
Put $\ell=\lambda N + \tilde{k}$ where $\lambda \in \zz$ and $0 \leq \tilde{k} \leq N-1$. 
We find $\lambda, \tilde{k}$ such that 
\begin{equation}\label{eq:20.1}
\frac{\ell \pi}{Np} - \theta_{p,L}(E_{0})= (\ld-m)\frac{\pi}{p} + \frac{\tilde{k}\pi +h_{j}(E_{0})}{Np} > 0.
\end{equation}
It is easy to see that, for $N$ large, the necessary condition is $\ld-m \geq -1$. Consider the case $\ld -m =-1$. Then, \eqref{eq:20.1} yields 
\begin{equation}\label{eq:20.2}
\tilde{k} \pi + h_{j}(E_{0}) > N \pi.
\end{equation}  
According to \cite[Lemma 4.7]{klopp13}, $h_{j}(E_{0}) \in \frac{\pi}{2} \zz$. We observe that if $h_{j}(E_{0})<0$, there does not exist $0 \leq \tilde{k} \leq N-1$ satisfying \eqref{eq:20.2}. Hence, $h_{j}(E_{0}) \in \frac{\pi}{2} \nn$. We distinguish two cases. First of all, assume that $h_{j}(E_{0}) \in \pi \nn$. Then, the first $\ell$ verifying \eqref{eq:20.1} and $\ld_{\ell} \in \ringz$ is $\ell_{0} = \frac{Np}{\pi} \theta_{p,L}(E_{0})+1$. Next, consider the case $h_{j}(E_{0}) \in \frac{\pi}{2}+ \pi \nn$. Then, the first $\ell$ chosen is $\ell_{0}=\frac{Np}{\pi} \theta_{p,L}(E_{0})+\frac{1}{2}$. Put $\ell_{k}= \ell_{0}+k$ and we associate $\ell_{k}$ to $\lambda_{k}$, the $(k+1)-$th eigenvalue in $\mathring{B_{i}}$. Then, we always have 
\begin{equation}\label{eq:shift}
\theta_{p,L}(\ld_{k})-\theta_{p,L}(E_{0})= \frac{(k+1)\pi}{L-j} + \frac{c_{0}}{L-j}
\end{equation}   
where $c_{0}=0$ if $h_{j}(E_{0}) \in \pi \zz$ and $c_{0}=-\frac{\pi}{2}$ otherwise.\\
Hence, \eqref{eq:114} and \eqref{eq:shift} yield $\ld_{k}-E_{0} \asymp \frac{(k+1)^{2}}{L^{2}}$ for all $\ld_{k} \in I$ with $\ve$ small and $L$ large. Consequently, $k \lesssim \ve (L-j)$.\\
Finally, we will prove the inequality \eqref{eq:survive}.\\
Recall that the functions $\theta_{p}(E_{0}+x^{2})$ and $h_{j}(E_{0}+x^{2})$ are analytic in $x$ on the whole band $B_{i}$. Then, we can expand these functions near $0$ to get 
$$\theta_{p}(E_{0}+x^{2}) = \theta_{p, 0}+ \theta_{p, 1}x +\theta_{p,2}x^{2} + O(x^{3}) \text{ where $\theta_{p, 0}=\theta_{p}(E_{0})$ and $\theta_{p, 1} \ne 0$}; $$
$$h_{j}(E_{0}+x^{2})= h_{j,0} +h_{j,1}x +h_{j,2}x^{2}+O(x^{3}) \text{ where $h_{j,0}=h_{j}(E_{0})$}.$$
Put $x_{k}=\sqrt{\ld_{k}-E_{0}}$. We can assume that $\theta_{p,L}$ is increasing on $B_{i}$. Then, \eqref{eq:shift} and the above expansions yield 
\begin{equation}\label{eq:300}
\theta_{p,1}(L) x_{k} + \theta_{p,2}(L) x^{2}_{k} +O(x_{k}^{3}) =\frac{(k+1) \pi}{L-j}+\frac{c_{0}}{L-j}
\end{equation}
where $\theta_{p, m}(L)=\theta_{p,m}- \frac{h_{j,m}}{L-j}$ for all $m\in \nn$, $c_{0}=0$ if $h_{j}(E_{0}) \in \pi \zz$ and $c_{0}=-\frac{\pi}{2}$ otherwise. Note that $|\theta_{p,1}(L)|$ is lower bounded and upper bounded by positive constants independent of $L$.\\
W.ol.g., assume that $c_{0}=0$. Then, we have 
\begin{equation}\label{eq:301}
x_{k}=  \tilde{c}(L) \cdot \frac{(k+1)\pi}{(L-j)} \cdot \frac{1}{1+ x_{k}g_{L}(x_{k})}  
\end{equation}
where $|\tilde{c}(L)|$ is lower bounded and upper bounded by positive constants independent of $L$. Moreover, the function $g_{L}$ is analytic near $0$; $g_L$ and its derivative are bounded near $0$ by constants independent of $L$.\\
Let $0\leq n<k \leq \ve (L-j)$, the equation \eqref{eq:301} yield 
\begin{align}\label{eq:302}
x_{k}-x_{n}= &\tilde{c}(L) \cdot \frac{\pi (k-n)}{L-j}\cdot \frac{1}{1+x_{k}g_{L}(x_{k})}+\\
& \tilde{c}(L)\cdot \frac{n+1}{L}  \cdot \frac{x_{n}g_{L}(x_{n}) - x_{k}g_{L}(x_{k})}{\left(1+x_{k}g_{L}(x_{k}) \right) \left(1+x_{n}g_{L}(x_{n}) \right)}. \notag
\end{align}
Note that the second term of the right hand side (RHS) of \eqref{eq:302} is bounded by $\ve |x_{k}-x_{n}|$ up to a constant factor. Hence, there exists a constant $C$ such that, for all $n<k\leq \ve(L-j)$, 
\begin{equation}\label{eq:303}
\frac{1}{C} \cdot \frac{k-n}{L-j} \leq |x_{k}-x_{n}| \leq C\cdot \frac{k-n}{L-j}. 
\end{equation}
On the other hand, $x_{k} \asymp \frac{k+1}{L-j}$ and $x_{n}\asymp \frac{n+1}{L-j}$. \\
We thus have $|\ld_{k}-\ld_{n}|=|x^{2}_{k}-x^{2}_{n}| \asymp \frac{|k^{2}-n^{2}|}{L^{2}}$ for all $0 \leq n<k \leq \ve (L-j)$.
\end{proof}
\begin{remark}\label{R:spacing}
For $L$ large, the average  distance between two consecutive, distinct eigenvalues (the spacing) is $\frac{1}{L}$. Lemma \ref{L:asymptoticformula} says that, the spacing between eigenvalues near $\partial \Sigma_{\zz}$ is much smaller, the distance  between $\lambda^{i}_k$ and $\lambda^{i}_{k+1} \in I=[E_0, E_0+\ve_{1}]$ where $\ve_{1} \asymp \ve^{2}$ has magnitude $\frac{k+1}{L^2}$. This fact implies that the number of eigenvalues in the interval $I$ is asymptotically equal to $\ve L$ as $L \rightarrow +\infty$.
\end{remark}
\section{Small imaginary part}\label{S:sip}
First of all, we prove the following lemma which will be useful when we estimate the sum $S_{L}(E)$.
\begin{lemma} \label{L:outsider}
Pick $\eta>0$ and $E_{0}\in \partial \Sigma_{\zz}$. For $E\in J:=\left[E_{0}, E_{0}+\eta \right]+i \rr$, we define $S_{\text{out}}(E) = \sum\limits_{|\lambda_{k}-E_{0}|> 2\eta} \frac{a_{k}}{\lambda_{k}-E}$
 Then, 
\begin{equation}\label{eq:outsider1}
|S_{\text{out}}(E)| \leq \frac{1}{\eta} \text{ and } |\text{Im} S_{\text{out}}(E)| \leq \frac{ |\text{Im} E|}{\eta^{2}}
\end{equation}
and
\begin{equation}\label{eq:outsider2}
0< S'_{\text{out}}(E) \leq \frac{1}{\eta^{2}} \text{ for all $E \in \left[E_{0}, E_{0} + \eta \right]$ }. 
\end{equation}
\end{lemma}
\begin{proof}[Proof of Lemma \ref{L:outsider}] 
Note that $|\lambda_{k}-E| > \eta$ for all $|\lambda_{k}-E_{0}|>2\eta$ and $E\in J$. On the other hand, 
$\text{Im} S_{\text{out}}(E) = \text{Im}E \sum\limits_{|\lambda_{k}-E_{0}|>2\eta} \frac{a_{k}}{|\lambda_{k}-E|^{2}}$ and 
$S'_{\text{out}}(E)=\sum\limits_{|\lambda_{k}-E_{0}|>2\eta} \frac{a_{k}}{(\lambda_{k}-E)^{2}}$. Hence, the claim follows.
\end{proof}

Now, we will prove that the imaginary part of $S_L(E)=\sum\limits_{k=0}^L \frac{a_k}{\lambda_k-E}$ is small if $|\text{Im} E|$ is not too small.
\begin{lemma}\label{L:lesson}
Assume the same hypothesis of the boundary point $E_0$ and we use the same enumeration for eigenvalues in the band $B_{i}$ containing $E_{0}$ as in Lemma \ref{L:asymptoticformula}.\\
Pick $\ve>0$ small, $C_{1}, L$ large and $0 \leq n \leq \ve L/C_{1}$ . Consider the domain $\cA_{n, \ve}=[\frac{\lambda^{i}_{n-1}+\lambda^{i}_{n}}{2}, \frac{\lambda^{i}_{n}+\lambda^{i}_{n+1}}{2}] -i \left[C_0\frac{n+1}{L^2}, \ve^5 \right]$ where $C_0$ is a large constant and $\lambda^{i}_{-1}:=2E_0-\lambda_0$. Then, for all $E\in \cA_{n, \ve}$ with $\ve$ sufficiently small, we have 
\begin{equation}\label{eq:lesson}
|\text{Im} S_L(E)| \lesssim \ve.
\end{equation}
As a result, there are no resonances in $\cA_{n, \ve}$.
\end{lemma}
\begin{proof}[Proof of Lemma \ref{L:lesson}]
Let $E= x-i y \in \cA_{n, \ve}$. Since Lemma \ref{L:asymptoticformula}, we can choose $C_{1}$ large enough so that $\lambda^{i}_{n} \leq E_{0}+\ve_{1}$ for all $n \leq \ve L/C_{1}$ and $\lambda^{i}_{k}>E_{0}+2\ve_{1}$ if $k>\ve L$ and $\lambda^{i}_{k} \in B_{i}$  where $\ve_{1}\asymp \ve^{2}$. Hence, Lemma \ref{L:outsider} yields that
\begin{equation}\label{eq:7.1}
|\text{Im} S_L(E)| \leq \frac{a^{i}_n y}{(\lambda^{i}_n-x)^2+y^2}+ \sum^{\ve L}_{\substack{k=0 \\ k\ne n}} \frac{a^{i}_k y}{(\lambda^{i}_k-x)^2+y^2}+ \ve
\end{equation}
where $\{a^{i}_{k}\}$ are $\{a_{k}\}$ renumbered w.r.t. the band $B_{i}$. For the sake of simplicity, we will skip the superscript $i$ in $\lambda^{i}_{k}$ and $a^{i}_{k}$ throughout the rest of proof.\\
Note that $\ld_{k} \asymp \frac{(k+1)^{2}}{L^{2}}$ and $a_{k} \asymp \frac{(k+1)^{2}}{L^{3}}$ for all $0 \leq k \leq \ve L$. Hence,
$$\frac{a_n y}{(\lambda_n-x)^2+y^2} \leq \frac{a_n}{y} \lesssim \frac{n+1}{L} \lesssim \ve.$$
Hence, it suffices to show that the sum 
\begin{equation}\label{eq:heart}
S= \sum^{\ve L}_{\substack{k=0 \\ k\ne n}}  \frac{a_k y}{(\lambda_k-x)^2+y^2} \lesssim \ve.
\end{equation}
To simplify our notations, from now on, we will not write $0 \leq k \leq \ve L$ in the sum. We upper bound $S$ as follows
\begin{equation}\label{eq:7.2}
S \leq  \sum_{\substack{k\ne n, \\ |\lambda_k-x| \leq y}}  \frac{a_k}{y} +\sum_{\substack{k\ne n, \\ |\lambda_k-x| \geq y}} \frac{a_k y}{(\lambda_k-x)^2} = S_1 +S_2.
\end{equation}
We will estimate $S_1$ first. For any index $k$ of the sum $S_1$, we have, for $C_0$ sufficiently large,
$$|\ld_k - \ld_n| \leq |\ld_k -x| + |x-\ld_n| \leq 2y.$$
Hence, $|k^2- n^2| \leq 2CL^2 y$ for some $C>0$. In the other words, 
$$(n^2- 2CL^2 y)_{+}  \leq k^2 \leq n^2 + 2CL^2 y$$
with $(x)_{+}=\max\{x, 0\}$. Hence, 
\begin{align} \label{eq:7.3}
S_1 &\lesssim \frac{1}{L^3 y} \sum_{k=\sqrt{(n^2- 2CL^2 y)_{+}}}^{\sqrt{n^2 + 2CL^2 y}} k^2 \lesssim \frac{1}{L^3 y}  \int_{\sqrt{(n^2- 2CL^2 y)_{+}}} ^{\sqrt{n^2 + 2CL^2 y}} x^2 dx \\ 
&\lesssim \frac{1}{L^3 y} \left( (n^2 + 2CL^2 y)^{3/2} -  (n^2- 2CL^2 y)_{+}^{3/2} \right) \notag.
\end{align}
If $n^2\leq 2CL^2 y$, the estimate \eqref{eq:7.3} yields that 
$$S_1 \lesssim \frac{1}{L^3 y} (L^2 y)^{3/2} \lesssim \sqrt{y}.$$
Otherwise, 
\begin{align*} 
S_1 \lesssim \frac{1}{L^3 y} \frac{(n^2 + 2CL^2 y)^{3} - (n^2- 2CL^2 y)^3}{(n^2 + 2CL^2 y)^{3/2} +  (n^2- 2CL^2 y)^{3/2}}
\lesssim \frac{1}{L^3 y} \frac{L^2 y n^4}{n^3} \lesssim \frac{n}{L}.
\end{align*}
To sum up, 
\begin{equation}\label{eq:7.4}
S_1 \lesssim \sqrt{y}+ \frac{n}{L}\lesssim \ve.
\end{equation}
Now, we will find a good upper bound for $S_2$. Repeating the argument as above, we only need to consider the sum w.r.t. to indices $k^2 \geq n^2 + \frac{1}{C} yL^2$ or $k^2 \leq n^2 -\frac{1}{C} y L^2$. We split and upper bound $S_2$ by two sums $S_3$ and $S_4$ which correspond to these two possibilities of index $k$.
\begin{align*}
S_3&= \sum_{k=\sqrt{n^2 + \frac{1}{C} yL^2}}^{\ve L} \frac{a_k y}{(\ld_k -x)^2} \lesssim \sum_{k=\sqrt{n^2 + \frac{1}{C} yL^2}}^{\ve L} \frac{a_k y}{(\ld_k -\ld_n)^2}\\
& \lesssim y L \sum_{k=\sqrt{n^2 + \frac{1}{C} yL^2}}^{\ve L} \frac{k^2}{(k-n)^2 (k+n)^2} \lesssim y L \sum_{k=\sqrt{n^2 + \frac{1}{C} yL^2}}^{\ve L} \frac{1}{(k-n)^2}\\ 
&\lesssim yL \sum_{k=\sqrt{n^2 + \frac{1}{C} yL^2}-n}^{\ve L-n} \frac{1}{k^2}\lesssim yL \frac{1}{\sqrt{n^2 + \frac{1}{C} yL^2}-n} \lesssim y  \frac{1}{\sqrt{\left(\frac{n}{L}\right)^2 + \frac{1}{C} y}-\frac{n}{L}}.
\end{align*}
Note that, for all $a, b>0$, $\frac{1}{\sqrt{a^2+b}-a} <  \sqrt{b}$. Hence, 
\begin{equation}\label{eq:7.5}
S_3 \lesssim y  \cdot \frac{\sqrt{y}}{\sqrt{C}} \leq \ve.
\end{equation}
Finally, we will estimate $S_4$. We only need to consider the case $y \leq C\frac{n^2}{L^2}$. Then,
\begin{align*}
S_4&= \sum_{k=0}^{\sqrt{n^2-\frac{1}{C}y L^2}} \frac{a_k y}{(\ld_k-x)^2} \lesssim \sum_{k=0}^{\sqrt{n^2-\frac{1}{C}y L^2}} \frac{a_k y}{(\ld_k-\ld_n)^2}\\
&\lesssim yL  \sum_{k=0}^{\sqrt{n^2-\frac{1}{C}y L^2}} \frac{k^2}{(n-k)^2(n+k)^2} \lesssim y L \sum_{n-\sqrt{n^2-\frac{1}{C}y L^2}}^n \frac{1}{k^2}.
\end{align*}
Note that $n-\sqrt{n^2-\frac{1}{C}y L^2}= \frac{\frac{1}{C} y L^2}{n+\sqrt{n^2-\frac{1}{C}y L^2}} \geq  \frac{yL^2}{2Cn}$. Hence, 
\begin{align}\label{eq:7.6}
S_4 \lesssim yL \sum_{\frac{yL^2}{2Cn}}^n \frac{1}{k^2} \lesssim yL \cdot \frac{n}{yL^2} \lesssim \ve.
\end{align}
Thanks to \eqref{eq:7.2} and \eqref{eq:7.4}-\eqref{eq:7.6}, the estimate \eqref{eq:heart}, hence, \eqref{eq:lesson} follows.\\ 
Note that, since $E_{0} \in (-2, 2)$, $\theta(E)= \arccos \frac{E}{2}$ is analytic and $|\sin (\text{Re} \theta(E))| \gtrsim 1$ near $E_{0}$. Consequently, for any $E\in [E_0, E_0+\ve_{1}] -i[0, \ve^5]$, there exists a constant $c_{0}>0$ such that
\begin{equation*}\label{eq:7.7}
|\text{Im} e^{-i \theta(E)}| = e^{\text{Im}\theta(E)} |\sin (\text{Re} \theta(E))| \geq c_0. 
\end{equation*}
 Hence, there are no resonances in $\cA_{n, \ve}$.
\end{proof}
\section{Resonances closest to the real axis}\label{S:proofmain}
In the present section, we will give a proof for Theorem \ref{T:maingeneric} which describes the resonances closest to the real axis. To do so, we will apply Rouch\'{e}'s theorem to show the existence and uniqueness of resonances in each rectangle $\mathcal M_n=\left[\frac{\lambda^{i}_{n-1}+\lambda^{i}_{n}}{2}, \frac{\lambda^{i}_{n}+\lambda^{i}_{n+1}}{2}\right] -i \left[0, C_0\frac{n+1}{L^2}\right]$ for $0\leq n \leq \ve L/C_{1}$ with $C_{0}, C_{1}>0$ large. Next, we derive the asymptotic formulae for resonances.\\
Corresponding to the case $n=0$, we will apply Rouch\'{e}'s theorem in the rectangle\\ 
$\left [E_{0}-\ve, \frac{\lambda^{i}_{0}+\lambda^{i}_{1}}{2}\right] -i \left[0, \frac{C_{0}}{L^2}\right]$ instead of $\left [E_{0}, \frac{\lambda^{i}_{0}+\lambda^{i}_{1}}{2}\right] -i \left[0, \frac{C_{0}}{L^2}\right]$. Next, we will prove that the unique resonance $z_{0}$ in this rectangle stays close to $\ld^{i}_0$ at a distance $\frac{1}{L^3}$. Consequently, there are no resonances in $[E_{0}-\ve, E_{0}] - i \left[0, C_0 \frac{n+1}{L^2} \right]$ and $z_{0}$ belongs to $\left [E_{0}, \frac{\lambda^{i}_{0}+\lambda^{i}_{1}}{2}\right] -i \left[0, \frac{C_{0}}{L^2}\right]$. Such a result is needed to study resonances below $\rr \backslash \Sigma_{\nn}$ in Section \ref{S:proofmain1}.\\ 
For that purpose, in Lemmata \ref{L:difference} and \ref{L:comparegeneric}, we will use a different convention for $\lambda^{i}_{-1}$ from that in Theorem \ref{T:maingeneric}. Concretely, we put $\lambda^{i}_{-1}:=2(E_{0}-\ve)-\lambda^{i}_{0}$ instead of $2E_0 -\ld^i_{0}$ in these lemmata.    

Note that, by Lemma \ref{L:outsider}, when we study the resonance equation near a boundary point $E_{0}$, only eigenvalues inside the spectrum and near $E_{0}$ need taking into account. In order to simplify the notation and the presentation, we will prove our results for $E_{0}=\inf \Sigma_{\zz}$. For an arbitrary $E_{0} \in \partial \Sigma_{\zz}$, all proofs work with tiny modifications. Note that when $E_0 =\inf \Sigma_{\zz}$ and if we ignore eigenvalues of $H_L$ outside $\Sigma_{\zz}$, on the band containing $E_{0}$, two enumerations of eigenvalues (a usual one with increasing order and the other w.r.t. to bands of $\Sigma_{\zz}$) are the same. From now on, we will skip the superscript $i$ in $\lambda^{i}_{k}, a^{i}_{k}$ and the sum $S^{i}_{n,L}(E)$ defined in Theorem \ref{T:maingeneric} can be written simply as \begin{equation}\label{eq:sum}
S_{n,L}(E)= \sum_{\substack{ k=0 \\
k\ne n}}^{L} \frac{a_k}{\ld_k-E}.
\end{equation}
In order to use Rouch\'{e}'s theorem, we will need two following useful lemmata.
Lemma \ref{L:lowerbound*} gives us an estimate on the sum $S_{n,L}(\ld_n)$ and Lemma \ref{L:difference} show that, in $\cM_n$, $S_{n,L}(E)$ can be approximated by $S_{n,L}(\ld_n)$ with a small error.\\
Let's take a look at the sum $S_{n,L}(\ld_{n})$. First of all, the part of the sum w.r.t. $k >\ve L$ is bounded by a constant depending only on $\ve$. Next, from the asymptotic of $a_{k}$ and $\ld_k$ near $\partial \Sigma_{\zz}$, it is easy to check that, in the absolute value, the sums $\sum\limits_{k=0}^{n-1}\frac{a_k}{\ld_k- \ld_n}$ and $\sum\limits_{k=n+1}^{\ve L}\frac{a_k}{\ld_k- \ld_n}$ are of the same order $\frac{n \ln n}{L} \xrightarrow{L\rightarrow +\infty} +\infty$ if $n$ is large ($n=\ve L$ for example). However, we note that these two sums have opposite signs. We can actually show that they will cancel each other out to become very small (smaller than $\ve$ up to a constant factor). To make such a cancellation effect appear, the results on the smoothness of spectral data near $\partial \Sigma_{\zz}$ in Section \ref{S:smoothness} will be used.
\begin{lemma}\label{L:lowerbound*}
Let $\ve>0$ small and $0 \leq n \leq \ve L/C_{1}$ with $C_{1}$ large.\\ 
For $E\in \mathbb C$, let $S_{n,L}(E)$ be defined as in \eqref{eq:sum}. Then,
\begin{equation}\label{eq:7.24}
|S_{n,L}(\ld_n)| \lesssim  \frac{1}{\ve^2}.
\end{equation}
\end{lemma}
\begin{proof}[Proof of Lemma \ref{L:lowerbound*}]
First of all, since Lemma \ref{L:asymptoticformula}, we can choose $C_{1}$ chosen to be large, we can assume that $\lambda_n\leq E_0+\ve_{1}$ and $\lambda_k \geq E_0+2\ve_1$ with some $\ve_1 \asymp \ve^2$ for $k \geq \ve L$. 
Hence, Lemma \ref{L:outsider} yield
\begin{equation}\label{eq:7.25}
S_{n,L}(\ld_n) \lesssim \left | \sum_{\substack{ k=0 \\
		k\ne n}}^{2n}\frac{a_k}{\ld_k-\ld_n}\right | + \sum_{k=2n+1}^{\ve L}\frac{a_k}{\ld_k-\ld_n}  + \frac{1}{\ve^{2}}.
\end{equation}
Next, we estimate the sum $T=\sum\limits_{k=2n+1}^{\ve L}\frac{a_k}{\ld_k-\ld_n}$. Recall that $|\ld_{n}-\ld_{k}| \asymp \frac{|k^{2}-n^{2}|}{L^{2}}$ and $a_{k} \asymp \frac{k^{2}}{L^{3}}$. Hence,
\begin{align} \label{eq:7.26}
T &\asymp \frac{1}{L}\sum_{k=2n+1}^{\ve L}\frac{k^2}{k^2-n^2} \asymp \frac{1}{L}(\ve L-2n-1) +\frac{n^2}{L}\sum_{2n+1}^{\ve L} \frac{1}{(k-n)(k+n)} \notag \\ 
&\asymp \ve +\frac{n}{L} \sum_{k=2n+1}^{\ve L} \left(\frac{1}{k-n}-\frac{1}{k+n}\right) \asymp \ve +\frac{n}{L} \left( \sum_{k=n}^{3n} \frac{1}{k}-\sum_{k=\ve L-n}^{\ve L+n} \frac{1}{k}\right)\\
&\asymp \ve + \frac{n}{L} \asymp \ve \notag.
\end{align}
Now we will show that the sum $ \left | \sum\limits_{\substack{ k=0 \\ k\ne n}}^{2n}\frac{a_k}{\ld_k-\ld_n} \right | \lesssim \ve$. We rewrite
\begin{align*}
S=  \sum\limits_{\substack{ k=0 \\
		k\ne n}}^{2n}\frac{a_k-a_n}{\lambda_k-\lambda_n} +a_n \sum\limits_{\substack{ k=0 \\
k\ne n}}^{2n}\frac{1}{\ld_k-\ld_n}=S_1 +S_2.
\end{align*}
First, we will estimate $S_1$. Thanks to Lemma \ref{L:nicefunction}, we have
\begin{equation}\label{eq:7.33}
|S_1| \leq 2C_{\text{lip}} \cdot \frac{n}{L} \lesssim \ve. 
\end{equation}
Second, we consider the sum $S_2$.
\begin{align}\label{eq:7.34}
S_2= a_n \left(\sum_{k=0}^{n-1} \left( \frac{1}{\lambda_k-\lambda_n}+\frac{1}{\lambda_{2n-k}-\lambda_n}\right) \right) =a_n \sum_{k=0}^{n-1}\frac{\lambda_k+\lambda_{2n-k}-2\lambda_n}{(\lambda_k-\lambda_n)(\lambda_{2n-k}-\lambda_n)}.
\end{align}
Assume that $L=Np+j$ with $0 \leq j \leq p-1$. By Lemma \ref{L:smoothness}, for each $k\leq 0\leq n-1$, we can write $\lambda_k= \psi\left(\frac{k}{L}\right)$ where  
$\psi(x)$ is a $C^2$ function near $E_{0}$. Moreover, its second derivative near $E_{0}$ is bounded by a constant independent of $L$. Hence, we can apply the Taylor's expansion of the order $2$ for the function $\psi(x)$ to get 
\begin{equation}\label{eq:7.35}
\lambda_k+\lambda_{2n-k}-2\lambda_n= \psi\left(\frac{k}{L}\right) + \psi\left(\frac{2n-k}{L}\right)-2 \psi\left(\frac{n}{L}\right)= O\left(\frac{(n-k)^2}{L^2}\right). 
\end{equation}
By \eqref{eq:7.34} and \eqref{eq:7.35}, we infer that 
\begin{align} \label{eq:7.37}
|S_2| &\lesssim \frac{n^2}{L} \sum_{k=0}^{n-1} \frac{(n-k)^2}{\left(n^2-k^2)((2n-k)^2-n^2\right)} \lesssim \frac{n^2}{L} \sum_{k=0}^{n-1} \frac{1}{(n+k)(3n-k)} \notag\\ 
&\lesssim \frac{n}{L} \left(\sum_{k=0}^{n-1}\frac{1}{n+k} + \sum_{k=0}^{n-1}\frac{1}{3n-k} \right) \lesssim \ve. 	
\end{align}
To sum up, thanks to \eqref{eq:7.33} and \eqref{eq:7.37}, we have $|S_{n,L}(\lambda_n)| \lesssim \frac{1}{\ve^2}$.
\end{proof}
For $E\in \mathcal M_{n}$, we compare $S_{n,L}(E)$ with $S_{n,L}(\lambda_{n})$.
\begin{lemma}\label{L:difference}
Pick $C_{1}, C_{0}$ large, $\ve>0$ small and $0\leq n \leq \ve L/C_{1}$.\\ Let $\mathcal M_n=\left[\frac{\lambda_{n-1}+\lambda_{n}}{2}, \frac{\lambda_{n}+\lambda_{n+1}}{2}\right] -i \left[0, C_0\frac{n}{L^2}\right]$ and $S_{n,L}(E)$ be defined as in Lemma \ref{L:lowerbound*}.\\ 
We use the convention $\ld_{-1}= 2(E_{0}-\ve) - \ld_{0}$.\\ 
Then, for all $E\in \cM_n$, we have 
\begin{equation}\label{eq:difference}
\left |  S_{n, L}(E) -S_{n, L}(\ld_n) \right| \lesssim L|E-\lambda_n| \lesssim \frac{n}{L}.
\end{equation}
\end{lemma} 
\begin{proof}[Proof of Lemma \ref{L:difference}]
By the definition of $S_{n,L}(E)$, we have
\begin{align}\label{eq:7.9}
\left |  S_{n, L}(E) -S_{n, L}(\ld_n) \right| &\leq |\ld_n-E| \sum_{k\ne n} \frac{a_k}{|\ld_k-E| |\ld_k-\ld_n|}\\ 
&\lesssim  |\ld_n-E| \sum_{k\ne n} \frac{a_k}{(\ld_k-\ld_n)^{2}} \notag .
\end{align}
First of all, we observe that 
\begin{align}\label{eq:7.10}
S_1= 
\sum_{k=0}^{n-1} \frac{a_k}{(\ld_k-\ld_n)^2}\lesssim L \sum_{k=0}^{n-1} \frac{(k+1)^2}{(n-k)^2(n+k)^2}
\lesssim  L \sum_{k=1}^{n} \frac{(n-k+1)^2}{k^2(2n-k)^2} \lesssim L.
\end{align}
Next, 
\begin{align}\label{eq:7.11}
S_2= 
\sum_{k=n+1}^{\ve L} \frac{a_k}{(\ld_k-\ld_n)^2} \lesssim L \sum_{k=n+1}^{\ve L} \frac{k^2}{(k-n)^2(k+n)^2}
\lesssim  L \sum_{k=1}^{\ve L -n} \frac{(k+n)^2}{k^2(k+2n)^2} \lesssim L .
\end{align}
Finally, put 
$S_3= 
\sum\limits_{k>\ve L} \frac{a_k}{(\ld_k-\ld_n)^{2}}.$ 
We can apply Lemma \ref{L:outsider} to infer that $S_{3}$ is bounded by a constant depending only on $\ve$.
Combining this with \eqref{eq:7.9}-\eqref{eq:7.11}, the claim follows.
\end{proof}
Now we will make use of the above lemma to show the existence and uniqueness of resonances in each rectangle $\cM_{n}$ with 
$0\leq n \leq \ve L/C$.
\begin{lemma}\label{L:comparegeneric}
Pick $C_{1}, C_{0}>0$ large, $\ve>0$ small and $0\leq n \leq \ve L/C_{1}$.
Assume that $\cM_n$ is the rectangle defined in Lemma \ref{L:difference} with the convention $\ld_{-1}=2(E_{0}-\ve) -\ld_{0}$. Let $f(E):=S_L(E)+e^{-i\theta(E)}$ and $g_n(E):= \frac{a_n}{\ld_n-E} + S_{n,L}(\ld_n)+e^{-i \theta(\ld_n)}$ where $S_{n, L}(E)$ is defined in \eqref{eq:sum}.\\
Then, $f$ and $g$ have the same number of zeros in $\cM_{n}.$ As a result, there is a unique resonance in $\cM_{n}$.
\end{lemma}
\begin{proof}[Proof of Lemma \ref{L:comparegeneric}]
Note that if $\ld_{k}$ is an eigenvalue of $H_{L}$ which stays outside $\Sigma_{\zz}$, it is exponentially close to one of isolated simple eigenvalues of $H^{+}_{0}$ or $H^{-}_{j}$ with $L=Np+j$ (see Theorem \ref{T:spectraldata}). Hence, we can choose $\ve$ to be sufficiently small such that $[E_{0}-\ve, E_{0}] \cap \sigma(H_{L})= \emptyset$ for all $L$ large.\\  
Consequently, $f$ and $g$ are holomorphic in $\cM_n$ for all $0 \leq n \leq \ve L/C_{1}$. Recall that, in the present lemma, $\cM_{0}=\left[E_{0}-\ve, \frac{\lambda_{0}+\lambda_{1}}{2}\right] -i \left[0, C_0\frac{n}{L^2}\right]$.\\
We will prove first that $f$ and $g$ have the same number of zeros in $\cM_n$.\\
 First of all, since Lemma \ref{L:difference}, for all $E\in \cM_n$, we have
\begin{align}\label{eq:7.14}
&|f(E)-g_n(E)| \leq \left |e^{-i \theta(E)} -e^{-i \theta(\ld_n)}\right | +|S_{n,L}(E)-S_{n,L}(\ld_n)| \notag \\
&\lesssim L|\ld_n-E|+ \left |e^{-i \theta(E)} -e^{-i \theta(\text{Re}E)}\right| + \left |e^{-i \theta(\text{Re}E)} -e^{-i \theta(\ld_n)}\right|\\
&\lesssim L|\ld_n-E| + |\text{Im}E|+|\text{Re}E-\ld_n| \lesssim L|\ld_n-E| \lesssim \frac{n+1}{L} \notag.
\end{align}
Next we will check that, on the boundary $\gamma_n=ABCD$ of $\cM_n$ (see Figure \ref{F:f5}), $|g_n(E)|$ is much larger than $\frac{n}{L}$, hence, much larger than $|f(E)-g_n(E)|$.\\
To do so, we estimate the imaginary part of $g_n(E)$,  
\begin{align}\label{eq:7.15}
|\text{Im}g_n(E)|&= \left |\frac{a_n \text{Im} E}{(\ld_n -\text{Re}E)^2 + \text{Im}^2E} -e^{\text{Im}\theta(\ld_n)} \sin(\text{Re}\theta(\ld_n)) \right|\\
& \geq e^{\text{Im}\theta(\ld_n)} |\sin(\text{Re}\theta(\ld_n))| - \underbrace{\left| \frac{a_n \text{Im} E}{(\ld_n -\text{Re}E)^2 + \text{Im}^2E}\right|}_{= P}.\notag
\end{align}
We will now upper bound $P$ on the boundary $\gamma_n=ABCD$ of $\cM_n$.\\
On the interval $AB$, $E$ is real, hence, $P=0$.\\ 
On $AD$, $\text{Re}E=\frac{\ld_{n-1}+\ld_n}{2}$. 
Then,
$$(\ld_n -\text{Re}E)^2 + \text{Im}^2E \gtrsim (\ld_n-\ld_{n-1})^2 \gtrsim \frac{(n+1)^2}{L^4}.$$ 
The same bound holds for $E\in BC$.\\ 
Note that, when $n=0$, $|\ld_{n}-\ld_{n-1}|=|\ld_{0}-(E_{0}-\ve)| \gtrsim \ve \gg \frac{1}{L^{2}}$.\\
Finally, consider the interval $CD$ with $|\text{Im}E|=C_0 \frac{n+1}{L^2}$,
$$(\ld_n -\text{Re}E)^2 + \text{Im}^2E \geq  \text{Im}^2E = C_0^2 \frac{(n+1)^2}{L^4}.$$
To sum up, on the curve $\gamma_n$, 
\begin{equation}\label{eq:7.16}
P \lesssim \frac{a_n L^4 |\text{Im}E| }{(n+1)^2} \lesssim \frac{n+1}{L}.
\end{equation}
From \eqref{eq:7.15}-\eqref{eq:7.16}, there exists a constant $c_{0}>0$ such that 
\begin{equation}\label{eq:7.17}
|g_n(E)| \geq |\text{Im}g_n(E)| \geq c_0 > |f(E)-g_n(E)| \text{ on $\gamma_n$}.
\end{equation}
Then, \eqref{eq:7.14}, \eqref{eq:7.17} and Rouch\'{e}'s theorem yield that $f$ and $g$ have the same number of zeros in $\cM_n$.\\
We see that $g_n(E)=0$ admits the unique solution $\tilde{z}_n$ in $\mathbb C$ given by:
\begin{equation}\label{eq:7.18}
\tilde{z}_n=\ld_n +\frac{a_n}{S_{n, L}(\ld_n)+e^{-i\theta(\ld_n)}}.
\end{equation}
Let's check that $\tilde{z}_n$ belongs to $\cM_n$. Note that, by our convention for $\theta(E)$, $\theta(\ld_n) \in [-\pi, 0]$, hence, $\text{Im}\tilde{z}_n$ is negative. Moreover, since $|\sin \left(\theta(\lambda_{n}) \right) | \geq c_{0}>0$, we have 
\begin{equation}\label{eq:7.19}
|\ld_n - \tilde{z}_n| \leq \frac{a_n}{|\sin(\theta(\ld_n))|} \leq \frac{a_n}{c_0} \lesssim \frac{(n+1)^2}{L^3}.
\end{equation}
Hence, $\tilde{z}_n \in \cM_n$. This implies that the equation $f(E)=0$ has a unique solution, say $z_n$, in $\cM_n$ as well. In the other words, $z_n$ is the unique resonance in $\cM_n$.
\end{proof}
Finally, we complete the present chapter by giving a proof for the main theorem, Theorem \ref{T:maingeneric}.  
\begin{proof}[Proof of Theorem \ref{T:maingeneric}]
First of all, Lemmata \ref{L:lesson} and \ref{L:comparegeneric} yield that there is one and only one resonance, say $z_{n}$,  in each rectangle $B_{n, \ve}=\left[ \frac{\lambda_{n-1}+\lambda_{n}}{2}, \frac{\lambda_{n}+\lambda_{n+1}}{2} \right] -i[0,\ve^{5}]$ for any $1 \leq n \leq \ve L/C_{1}$. For $n=0$, there is a unique resonance $z_{0}\in \left [E_{0}-\ve, \frac{\lambda_{0}+\lambda_{1}}{2}\right] -i \left[0, \frac{C_{0}}{L^2}\right]$. Recall that we use the convention $\ld_{-1}:=2E_{0}-\ld_{0}$ in Theorem \ref{T:maingeneric}. Then, $\frac{\ld_{-1}+\ld_0}{2}$ is equal to $E_0$, not $E_0-\ve$. We will prove later that $z_0$ actually stays inside the rectangle $\cM_0= \left [E_{0}, \frac{\lambda_{0}+\lambda_{1}}{2}\right] -i \left[0, \frac{C_{0}}{L^2}\right]$. \\
We will now take one step further to say something about the magnitude of $z_n$ and its imaginary part.
Let $\tilde{z}_{n}$ be the number defined in \eqref{eq:7.18}. Put $\alpha_{n}= S_{n,L}(\lambda_{n}) +e^{-i \theta(\lambda_{n})}$.
Then, since Lemma \ref{L:lowerbound*}, $|\text{Im}\alpha_{n}| = |\sin \left(\theta(\lambda_{n}) \right) | \geq c_{0}>0$ and $|\alpha_{n}| \lesssim \frac{1}{\ve^{2}}$.\\
Let's consider the square $D_{n,r}=\tilde{z}_n+ r[-1, 1]^2$ centered at $\tilde{z}_n$. We will choose $r< \frac{a_{n}}{|\alpha_{n}|}$ such that we can make sure that the resonance $z_n$ belongs to $D_{n,r}$ by Rouch\'{e} theorem. Precisely, we find $r$ such that
\begin{equation}\label{eq:7.20}
|g_n(E)| > |f(E)-g_n(E)| \text{ on the boundary of $D_{n,r}$}.
\end{equation} 
First, we rewrite $g_n(E)$ as follows
\begin{equation}\label{eq:7.21}
|g_n(E)| =|g_n(E)-g_n(\tilde{z}_n)| = |E-\tilde{z}_n|\frac{a_n}{|\ld_n -E||\ld_n-\tilde{z}_n|}=|\alpha_{n}| \frac{|E-\tilde{z}_{n}|}{|\ld_{n}-E|}.
\end{equation}
Note that, for all $E\in \partial D_{n,r}$, $\frac{r}{2}\leq |E-\tilde{z}_n| \leq \frac{r}{\sqrt{2}} $. Combining this and \eqref{eq:7.19}, we infer that 
$$ |\ld_n -E| \leq |\ld_n-\tilde{z}_n| +|E-\tilde{z}_n| \leq \frac{a_{n}}{|\alpha_{n}|} +\frac{r}{\sqrt{2}} \leq 2 \frac{a_{n}}{|\alpha_{n}|}.$$
Hence, 
\begin{equation}\label{eq:7.22}
|g_n(E)| \geq \frac{|\alpha_{n}|^{2}}{4 a_{n}} r.
\end{equation}
On the other hand, from \eqref{eq:7.14}, for all $E\in \partial D_{n,r}$,
\begin{equation}\label{eq:7.23}
|f(E)-g_n(E)| \lesssim L|\ld_n-E| \lesssim L (|\ld_n-\tilde{z}_n|+|\tilde{z}_n-E|) \lesssim L \frac{a_{n}}{|\alpha_{n}|}.
\end{equation}
Hence, it suffices to choose $r< \frac{a_{n}}{|\alpha_{n}|}$ such that $\frac{|\alpha_{n}|^{2}}{4 a_{n}} r \geq C L \frac{a_{n}}{|\alpha_{n}|} $ where $C$ is a large constant.
Obviously, $r= \frac{C}{|\alpha_{n}|^{3}}\cdot \frac{(n+1)^4}{L^5}$ satisfies the above inequality with $C$ large.\\
Hence, by Rouch\'{e}'s theorem, the resonance $z_n$ belongs to $D_{n,r}$ and 
\begin{equation}\label{eq:7.70}
\left | z_{n}- \ld_{n}- \frac{a_{n}}{\alpha_{n}} \right| \leq \frac{C}{|\alpha_{n}|^{3}} \cdot \frac{(n+1)^{4}}{L^{5}}.
\end{equation}
Hence, the asymptotic formula \eqref{eq:asymptotic*} follows.\\
We  now estimate the imaginary part of $z_{n}$. Since \eqref{eq:7.70}, we have 
\begin{equation}\label{eq:7.71}
\left | \text{Im}z_{n}- \frac{a_{n} \sin\left( \theta(\ld_{n})\right)  }{|\alpha_{n}|^{2}} \right| \leq \frac{C}{|\alpha_{n}|^{3}} \cdot \frac{(n+1)^{4}}{L^{5}}.
\end{equation} 
Consequently, the asymptotic formula \eqref{eq:imEn*} for $\text{Im}z_n$ holds true and $|\text{Im} z_n| \lesssim \frac{(n+1)^2}{L^3}$.  Moreover, there exists $C>0$ such that
\begin{equation}\label{eq:7.72}
\left| \text{Im} z_{n}\right| \geq  \frac{(n+1)^{2}}{|\alpha_{n}|^{2} L^{3}} \left( \frac{1}{C}- \frac{C (n+1)^{2}}{|\alpha_{n}| L^{2}} \right) \gtrsim \ve^{4} \frac{(n+1)^{2}}{L^{3}}.
\end{equation}
Finally, when $n=0$, \eqref{eq:7.70} yields $|z_{0}-\ld_{0}| \lesssim \frac{1}{L^{3}}$. Hence, $z_{0}$ belongs to the rectangle $\left [E_{0}, \frac{\lambda^{i}_{0}+\lambda^{i}_{1}}{2}\right] -i \left[0, \frac{C_{0}}{L^2}\right]$. On the other hand, $z_{0}$ is the unique resonance in the rectangle $\left [E_{0}-\ve, \frac{\lambda^{i}_{0}+\lambda^{i}_{1}}{2}\right] -i \left[0, \frac{C_{0}}{L^2}\right]$. As a result, there are no resonances in $\left [E_{0}-\ve, E_{0}\right] -i \left[0, \frac{C_{0}}{L^2}\right]$.\\
We thus have Theorem \ref{T:maingeneric} proved.
\end{proof}
\section{Outside the spectrum}\label{S:proofmain1}
In the present section, we give a proof for Theorem \ref{T:outside} which describes a free resonance region of $H_{L}^{\nn}$ below intervals which meet $\partial \Sigma_{\zz}$ from outside the spectrum $\Sigma_{\zz}$.
\begin{proof}[Proof of Theorem \ref{T:outside}]
First of all, according to the point $(1)$ of Theorem \ref{T:maingeneric}, there are no resonances in $\left[E_{0}-\ve, E_{0} \right]- i \left[0, \frac{C_{0}}{L^{2}} \right]$ with $C_{0}>0$ large. Next, we will show that there are no resonances in $\cR_{1}= \left[E_{0}-\ve, E_{0} \right]- i \left[\frac{C_{0}}{L^{2}}, \ve^{5} \right]$ either. In order to do so, it suffices to prove that  
\begin{equation}\label{eq:7.73}
|\text{Im}S_{L}(E)| \lesssim \ve \text{ in $\cR_{1}$}.
\end{equation}
Note that $S_{L}(E)$ is holomorphic in the domain $\mathring{\cR_{1}}$.\\ 
Hence, $|\text{Im}S_{L}(E)|=-\text{Im}S_{L}(E)$ is a harmonic function in $\mathring{\cR_{1}}$. By the maximum principle for harmonic functions, it thus suffices to prove \eqref{eq:7.73} on the boundary $\gamma= ABCD$ of $\cR_{1}$ (see Figure \ref{F:f6}).\\
Let $(\ld^{i}_{\ell})_{\ell}$ be (distinct) eigenvalues of $H_{L}$ in the band $B_{i}$. Reasoning as in Lemma \ref{L:lesson}, we can assume that $\ld^{i}_{k}>E_{0}+ 2\ve^{2}$ for all $k>\ve L$ and $\ld^{i}_{k}\in B_{i}$ to get 
\begin{equation}\label{eq:7.74}
|\text{Im} S_L(E)| \lesssim \sum^{\ve L}_{k=0} \frac{a^{i}_k y}{(\lambda^{i}_k-x)^2+y^2}+ \ve \text { for all $z=x-iy \in \cR_{1}$ with $y>0$}.
\end{equation}
Throughout the rest of the proof, we will skip the superscript $i$ in $\ld^{i}_{k}$ and $a^{i}_{k}$.\\
From \eqref{eq:7.74}, it suffices to show that 
\begin{equation}\label{eq:7.75}
S= \sum^{\ve L}_{k=0}  \frac{a_k y}{(\lambda_k-x)^2+y^2} \lesssim \ve \text{ on $\gamma=ABCD$}.
\end{equation}
First of all, we consider $S$ on $AB$. On the interval $AB$, $y=\frac{C_{0}}{L^{2}}$  and $x\in [E_{0}-\ve, E_{0}]$. Then,
$|\ld_{k} -x| \geq |\ld_{k}-E_{0}| \gtrsim \frac{k^{2}}{L^{2}} \geq y $ for all $\ve L \geq k \gtrsim \sqrt{C_{0}}$. Combining this with the fact that $a_k \asymp \frac{|\ld_k-E_0|}{L} \asymp \frac{k^2}{L^3}$, we have
\begin{align}\label{eq:7.77}
\sum_{\ve L \geq k \geq \sqrt{C_{0}}}  \frac{a_k y}{(\lambda_k-x)^2+y^2}  \asymp \frac{1}{L^{2}} \sum_{\ve L \geq k \geq \sqrt{C_{0}}}  \frac{a_k}{(\lambda_k-x)^2} \lesssim \frac{1}{L^{2}} \sum_{\ve L \geq k \geq \sqrt{C_{0}}} \frac{L}{k^{2}} \lesssim \frac{1}{L}.
\end{align}
On the other hand, we see that 
\begin{align}\label{eq:7.76}
\sum^{\sqrt{C_{0}}}_{k=0}  \frac{a_k y}{(\lambda_k-x)^2+y^2} \leq \frac{1}{y} \sum^{\sqrt{C_{0}}}_{k=0} a_{k} \lesssim 
\frac{1}{L} \sum^{\sqrt{C_{0}}}_{k=1} k^{2} \lesssim \frac{1}{L}.
\end{align} 
Therefore, \eqref{eq:7.77}-\eqref{eq:7.76} yield $S \lesssim \frac{1}{L} \ll \ve$ on $AB$.\\
Next, on $CD$, $y=\ve^{5}$ and $x\in [E_{0}-\ve, E_{0}]$. We will split $S$ into two sums $S_{1}$ and $S_{2}$. On the one hand, we estimate
\begin{equation}\label{eq:7.78}
S_{1}= \sum^{\ve^{2} L-1}_{k=0} \frac{a_k y}{(\lambda_k-x)^2+y^2} \leq \frac{1}{y} \sum^{\ve^{2} L-1}_{k=0} a_{k}  \lesssim 
\frac{1}{\ve^{5}L^{3}} \sum^{\ve^{2} L-1}_{k=1} k^{2} \lesssim  \frac{1}{\ve^{5}L^{3}} \cdot (\ve^{2} L)^{3} \lesssim \ve.
\end{equation}
On the other hand, 
\begin{align}\label{eq:7.79}
S_{2}&= \sum^{\ve L}_{k=\ve^{2}L} \frac{a_k y}{(\lambda_k-x)^2+y^2} \leq y \sum^{\ve L}_{k=\ve^{2}L}\frac{a_k}{(\lambda_k-x)^2} \leq \ve^{5} L  \sum^{\ve L}_{k=\ve^{2}L} \frac{1}{k^{2}} \notag \\
&\lesssim \ve^{5} L  \left( \frac{1}{\ve^{2}L}- \frac{1}{\ve L}\right) \lesssim \ve^{3}.  
\end{align}
By \eqref{eq:7.78} and \eqref{eq:7.79}, we infer that $S \lesssim \ve$ on $CD$.\\
Note that $|S| \lesssim \ve$ on $BC$ according to Lemma \ref{L:lesson}. Finally, we consider the sum $S$ on the interval $AD$ where $z=E_{0}-\ve -iy$ with $\frac{C_{0}}{L^{2}} \leq y \leq \ve^{5}$. In this case, $|\ld_{k}-x| \geq |x-E_{0}| - |\ld_{k}-E_{0}| \gtrsim \ve$ for all $0\leq k \leq \ve L$. Hence, 
\begin{equation*}\label{eq:7.80}
S \asymp  y \sum^{\ve L}_{k=0} \frac{a_k}{(\lambda_k-x)^2} \lesssim \frac{y}{\ve^{2}} \sum^{\ve L}_{k=0} a_{k} \lesssim \ve^{3}  \sum^{\ve L}_{k=1} \frac{k^{2}}{L^{3}} \lesssim \ve^{3} \cdot \frac{(\ve L)^{3}}{L^{3}} \lesssim \ve^{6}<\ve.
\end{equation*}
To sum up, $S$, hence $|\text{Im} S_{L}(E)|$, is bounded by $\ve$ up to a positive constant factor on $\cR_{1}$. Therefore there are no resonances in $\cR_{1}$ and the claim follows.
\end{proof}
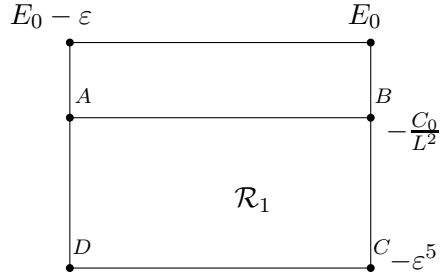
\begin{figure}[H]
\begin{center}
\begin{tikzpicture}[line cap=round,line join=round,x=0.5cm,y=0.5cm]
\clip(-1.39,-1.0) rectangle (13.79, 9.13);
\draw (2,6)-- (10,6);
\draw (10,6)-- (10,4);
\draw (10,4)-- (2,4);
\draw (2,4)-- (2,6);
\draw (2,0)-- (10,0);
\draw (10,0)-- (10,4);
\draw (2,4)-- (2,0);
\draw (0.13,7.3) node[anchor=north west] {$E_0 -\ve$};
\draw (9.12,7.3) node[anchor=north west] {$E_0$};
\draw (10.13,4.41) node[anchor=north west] {$- \frac{C_0}{L^2}$};
\draw (10.23,0.9) node[anchor=north west] {$-\varepsilon^5$};
\draw (6.11,2.41) node[anchor=north west] {$\mathcal R_1$};
\begin{scriptsize}
\fill [] (2,6) circle (1.5pt);
\fill [] (10,6) circle (1.5pt);
\fill [] (10,4) circle (1.5pt);
\draw[] (10.34,4.59) node {$B$};
\fill [] (2,4) circle (1.5pt);
\draw[] (2.36,4.59) node {$A$};
\fill [] (2,0) circle (1.5pt);
\draw[] (2.32,0.57) node {$D$};
\fill [] (10,0) circle (1.5pt);
\draw[] (10.3,0.57) node {$C$};
\end{scriptsize}
\end{tikzpicture}
\caption[]{Free resonance region below $[E_{0}-\ve, E_{0}]$}
\label{F:f6}
\end{center}
\end{figure}


\def\cprime{$'$}

\noindent {\tiny (Trinh Tuan Phong) Laboratoire Analyse, G\'{e}om\'{e}trie $\&$ Applications,\\ 
	UMR 7539, Institut Galil\'{e}e, Universit\'{e} Paris 13, Sorbonne Paris Cit\'{e},\\ 
	99 avenue J.B. Cl\'{e}ment, 93430 Villetaneuse, France\\
\emph{Email}: trinh@math.univ-paris13.fr 
}

\end{document}